\documentclass[12pt, reqno]{amsart}

\numberwithin{equation}{section}
\usepackage{enumerate}
 \DeclareMathOperator{\Ind}{Ind}

 \usepackage{geometry }
 \usepackage{float}
 
\usepackage{varioref}
\theoremstyle{plain}
\usepackage{color}
\usepackage{amssymb,amsmath,amscd}
\usepackage{amsthm}
 \numberwithin{dummy}{section}

\newtheorem{theorem}{Theorem}[section]
\newtheorem{theorem*}{Theorem }
\newtheorem{proposition}[theorem]{Proposition}

\newtheorem{lemma}[theorem]{Lemma}
\newtheorem{remark}[theorem]{Remark}

\DeclareMathOperator{\Ad}{Ad}

\DeclareMathOperator{\id}{id}
\DeclareMathOperator{\Id}{Id}

\DeclareMathOperator{\Lie}{Lie}

\DeclareMathOperator{\res}{\mathbf{res}}

 \DeclareMathOperator{\Rel}{Re}
 \DeclareMathOperator{\dd}{\boldsymbol d}
\DeclareMathOperator{\cdd}{\boldsymbol \delta}
\DeclareMathOperator{\di}{\boldsymbol \iota}
\DeclareMathOperator{\de}{\boldsymbol \varepsilon}
  
 \usepackage{mathtools}
 \usepackage{lipsum}
\usepackage[bbgreekl]{mathbbol}

\DeclareSymbolFontAlphabet{\mathbb}{AMSb}
\DeclareSymbolFontAlphabet{\mathbbl}{bbold}
 
 \usepackage[
 pdftex=true, colorlinks=true,  citecolor=cyan,   
 citebordercolor={0 .255 .255} , 
  linkbordercolor={0 .255 .255},  
  linkcolor=cyan
 ]{hyperref}

   \usepackage{cite}

 \def\equationautorefname~#1\null{(#1)\null}
 
\usepackage{multirow}
 \usepackage{stmaryrd}
 \usepackage{graphicx}
 \usepackage{caption}
 \usepackage{booktabs}
 \usepackage{threeparttable}

\makeatletter
\ifcsname phantomsection\endcsname
    \newcommand*{\qrr@gobblenexttocentry}[5]{}
\else
    \newcommand*{\qrr@gobblenexttocentry}[4]{}
\fi
\newcommand*{\addsubsection}{%
    \addtocontents{toc}{\protect\qrr@gobblenexttocentry}%
    \subsection}
\makeatother

\usepackage{lipsum}
\usepackage{mathabx}
\usepackage[normalem]{ulem}

\begin{document}
 

\title[covariant bi-differential operators for
differential forms]{Conformally covariant bi-differential operators for
differential forms}

\author{Salem Ben Sa\"id, Jean-Louis Clerc and Khalid Koufany }
\address{Institut Elie Cartan de Lorraine, UMR CNRS 7502, Universit\'e de Lorraine, Campus Aiguillettes - BP 70239, F-54506 Vandoeuvre-l\`es-Nancy Cedex, France}
\email{(Salem.Bensaid / Jean-Louis.Clerc / Khalid.Koufany) @univ-lorraine.fr}
\address{Salem Ben Said ({\it Current address}):  Department of Mathematical Sciences, College of Science, United Arab Emirates University, Al Ain Abu Dhabi, UAE. }
\keywords{Differential forms, conformal group, covariant bi-differential operators,   principal series, Riesz distribution}
\subjclass[2000]{Primary 43A85. Secondary 58J70, 22E46, 58A10}


 \maketitle
 
  \begin{abstract} The classical  Rankin-Cohen brackets  are bi-differential operators 
 from $C^\infty(\mathbb R)\times  C^\infty(\mathbb R)$ into $ C^\infty(\mathbb R)$. They are  covariant for the (diagonal) action of ${\rm SL}(2,\mathbb R)$ through principal series representations. We construct generalizations of these operators, replacing $\mathbb R$ by $\mathbb R^n,$ the group ${\rm SL}(2,\mathbb R)$ by the group ${\rm SO}_0(1,n+1)$ viewed as the conformal group of $\mathbb R^n,$  and functions by  differential forms. 
   \end{abstract}
   \tableofcontents
   \section{Introduction}
   The \emph{Rankin-Cohen brackets} are the most famous examples of conformally covariant bi-differential operators. For a presentation of these operators from our point of view based on harmonic analysis of the group ${\rm SL}(2,\mathbb R)$, we refer the reader to the introduction of \cite{bck}. In \cite{or} Ovsienko and Redou introduced their analogs for conformal analysis on $\mathbb R^n$.

A new construction of these covariant bi-differential operators was proposed by Beckmann and the second present author in \cite{bc}, where (although implicitly) the \emph{source operator} method was  introduced. In our situation, the source operator is a differential operator on $\mathbb R^n\times \mathbb R^n$, covariant for the diagonal action of the conformal group ${\rm SO_0}(1,n+1)$. The covariant bi-differential operators are obtained by  composing the source operator with the restriction map from $\mathbb R^n\times \mathbb R^n$ to the diagonal. This technique has shown to be very efficient to produce new examples of covariant differential operators in many different contexts.  In \cite{bck}, we constructed covariant bi-differential operators in the context of \emph{simple real Jordan algebras}.  
The article  \cite{c} contains an alternative construction of the covariant differential operators  introduced by   Juhl \cite{j} in the context of the  restriction of $\mathbb R^n$ to $\mathbb R^{n-1}$. In the same geometric context,   Fischmann,  \O rsted and  Somberg \cite{fos}  recently obtained a new construction of the covariant differential operators for differential forms, previously obtained by  Kobayashi,  Kubo and   Pevzner in \cite{kkp} and  by   Fischmann,   Juhl and   Somberg in \cite{fjs}. Finally, it is worthwhile mentioning that a more general notion of \emph{symmetry breaking operators} (not necessarily differential) has been studied by  Kobayashi and his collaborators (see, e.g., \cite{ks1}, \cite{ks2}, \cite{kp}).

In the present paper, we construct bi-differential operators acting on spaces of  differential forms which are covariant for the conformal group of $\mathbb R^n$; more precisely for the group  $G={\rm SO_0}(1,n+1)$. To build these bi-differential operators, we use again the  \emph{source operator} method. The source operators are constructed as a composition of the multiplication operator  by the function $\Vert x-y\Vert^2$ (using its transformation rule under the action of the conformal group) and  classical Knapp-Stein intertwining operators. These intertwining operators on differential forms  were studied recently in   \cite{fo, fos} and some of their results are used (and sometimes reproved) in the present article.

The construction relies ultimately on two \emph{main  identities} stated in Theorem \ref{main1} and Theorem \ref{main2} (see also Theorem \ref{normal}),    the second one being the Euclidean Fourier transform of the first one.  As they involve purely Euclidean harmonic analysis they are presented in Section 3, independently of the conformal context.   In Section 4 we give some background on the conformal group of $\mathbb{R}^n$ needed to  describe the {\em noncompact model} for the principal series representations of $\rm{SO}_0(1,n+1)$ in Section 5.   The conformal properties of the source operator are  given in Section  6, where harmonic analysis of the group ${\rm SO_0}(1,n+1)$ plays a crucial role.

The corresponding covariant bi-differential operators are constructed in  Section 7. The lack of a manageable decomposition of the tensor product $\sigma_k\otimes \sigma_\ell$, $0\leq k,\ell\leq n,$ where $\sigma_k$ denotes the representation of ${\rm SO}(n)$ on the spacecomplex-valued  $\Lambda^k$ of complex-valued  alternating $k$-forms on $\mathbb{R}^n$,  prevents us from giving explicit formulas    for the corresponding bi-differential operators, but this can be done at least for the Cartan factor $\Lambda^{k+\ell}$ with  $0\leq k+\ell \leq n$, appearing in $\Lambda^k\otimes \Lambda^\ell$,  (see \autoref{dernier}).

 \section{Background on differential forms}
 
 
 Let $\langle \cdot,\cdot\rangle$ be the standard Euclidean scalar product in $\mathbb R^n$    and $\|\cdot\|$ the corresponding norm. Let 
 $(e_1,e_2,\ldots, e_n)$ be the standard orthonormal  basis of  $\mathbb{R}^n$   and let $(e_1^*,e_2^*,\ldots, e_n^*)$ be its dual basis. 
 
  
  
 For $0\leq k\leq n,$ we denote by $ \Lambda^k=\Lambda^k({\mathbb R}^{{n}^*})\otimes\mathbb{C}$ the vector space of complex-valued  alternating multilinear $k$-forms on $\mathbb R^n.$  
 A basis of the space $\Lambda^k $ is given by 
 $$\{ e_I^*:=e_{i_1}^*\wedge \cdots \wedge e_{i_k}^*\;:\;  1\leq i_1<\cdots <i_k\leq n\}.$$
 If $\omega \in \Lambda^k$ and $\eta\in \Lambda^\ell,$ then $\omega \wedge \eta\in \Lambda^{k+\ell}.$ 
 Furthermore, 
 \begin{equation}\label{kl}
   \omega\wedge \eta=(-1)^{k\ell} \,\eta\wedge \omega.
   \end{equation}
The exterior algebra  $\Lambda:=\bigoplus_{k=0}^\infty  \Lambda^k=\bigoplus_{k=0}^n  \Lambda^k$  is an associative algebra, graded  with respect to the degree $k$.

 The interior product of a $k$-form $\omega$ with a vector $x\in \mathbb R^n$ is the $(k-1)$-form defined by 
$$
  {\boldsymbol \iota}_x \omega(x_1, \ldots, x_{k-1})=\omega(x, x_1, \ldots, x_{k-1}).
$$
Moreover, 
\begin{equation}\label{cases} {\boldsymbol \iota}_{e_j} (e^*_{i_1}\wedge \cdots \wedge e^*_{i_k})= \begin{cases}
  0 &\text{  if} \; j \not = \text{any}\; i_r \\
  (-1)^{r-1} e^*_{i_1}\wedge \cdots \wedge \widehat{e^*_{i_r}}\wedge \cdots \wedge e^*_{i_k}&\text{  if} \; j=i_r,
  \end{cases}\end{equation}
  where the ``cap" over $e^*_{i_r}$ means that it is deleted from the exterior product. 
 One may check that \begin{equation}\label{commi}
{\boldsymbol \iota}_x {\boldsymbol \iota}_y+{\boldsymbol \iota}_y {\boldsymbol \iota}_x =0.\end{equation}

 Given $x\in \mathbb R^n,$ the exterior product of a $k$-form $\omega$ with the linear form $x^*$ is the $(k+1)$-form defined by 
 \begin{equation*}
{\boldsymbol \varepsilon}_x\omega=x^*\wedge \omega.
  \end{equation*} 
  From the associativity of the wedge product and \autoref{kl},  it follows
\begin{equation}\label{comm}{\boldsymbol \varepsilon}_x {\boldsymbol \varepsilon} _y + {\boldsymbol \varepsilon}_y {\boldsymbol \varepsilon}_x =0.
\end{equation}
There is the following useful anticommutation relation
\begin{equation}\label{ID}
 {\boldsymbol \varepsilon}_x {\boldsymbol \iota}_y +{\boldsymbol \iota}_y {\boldsymbol \varepsilon}_x = \langle x,y\rangle \Id_{\Lambda},\qquad x,y\in \mathbb R^n.
 \end{equation}  This is a fairly straightforward consequence of \autoref{cases}.  Further, by \autoref{ID}, \autoref{comm} and \autoref{commi} we   have
  $${\boldsymbol \varepsilon}_x {\boldsymbol \iota}_y {\boldsymbol \varepsilon}_x  =\langle x,y\rangle    {\boldsymbol \varepsilon}_x ,\qquad
   {\boldsymbol \iota}_y {\boldsymbol \varepsilon}_x {\boldsymbol \iota}_y  =\langle x,y\rangle   {\boldsymbol \iota}_y .$$
 
 We denote by ${\boldsymbol \iota}_{ j } $  the interior    product  with the basis vector $e_j$   and by ${\boldsymbol \varepsilon}_{ j } $   the exterior products with $e_j^*.$  
  The following lemma is needed for  later use.

\begin{lemma} On $\Lambda^k  $ we have
\begin{equation*}
\sum_{j=1}^n {\boldsymbol \varepsilon}_j{\boldsymbol \iota}_j  =k \Id_\Lambda,\qquad
\sum_{j=1}^n {\boldsymbol \iota}_j {\boldsymbol \varepsilon}_j=(n-k) \Id_\Lambda.
\end{equation*}
\end{lemma}
\begin{proof} Let $I=\{i_1, \ldots, i_k\}\subset \{1,2,\dots, n\}$. In view of \autoref{cases} and the fact that $e^*_j\wedge e^*_j = 0$ for every $j$,   clearly we have
$$\Big(\sum_{j\in I}{\boldsymbol \varepsilon}_j{\boldsymbol \iota}_j\Big) e^*_I = k e^*_I, \qquad \Big(\sum_{j\in I} {\boldsymbol \iota}_j {\boldsymbol \varepsilon}_j \Big)e^*_I = 0.$$
Similarly,
$$\Big(\sum_{j\notin I}{\boldsymbol \varepsilon}_j{\boldsymbol \iota}_j\Big)e^*_I  = 0, \qquad \Big(\sum_{j\notin I} {\boldsymbol \iota}_j {\boldsymbol \varepsilon}_j \Big)e^*_I = (n-k) e^*_I.$$  The lemma is now a matter of  putting pieces together.
\end{proof}

From now on, we will identify the dual of $\mathbb{R}^n$ with $\mathbb{R}^n$. 
For $0\leq k\leq n$, let 
$$\mathcal{E}^k(\mathbb{R}^n)=C^\infty\left(\mathbb{R}^n, \Lambda^k \right)\simeq C^\infty(\mathbb{R}^n)\otimes \Lambda^k(\mathbb R^n)$$ 
be the space of smooth complex-valued differential forms of degree $k$ on $\mathbb{R}^n$. An  element of  $\mathcal{E}^k(\mathbb{R}^n)$ can be uniquely represented as
\begin{equation}\label{repkfo}
\omega (x)=\sum_{1\leq i_1<\cdots<i_k\leq n} \omega_{i_1,\cdots,i_k}(x) e_{i_1,\cdots,i_k},
\end{equation}
where the coefficients $\omega_I$ are   complex-valued smooth functions on $\mathbb R^n.$  
 In particular, $\mathcal E^0(\mathbb R^n) =C^\infty(\mathbb R^n) $. The direct sum $$\mathcal E (\mathbb R^n) :=  \bigoplus_{k=0}^n  \mathcal E^k(\mathbb R^n) $$ is the linear space of all smooth differential forms. 

From the properties of the interior product $  {\boldsymbol \iota}_x$ and the exterior product ${\boldsymbol \varepsilon}_x$ on $\Lambda^k$ which we discussed above, one gets analogue properties on $\mathcal E^k(\mathbb R^n).$
 
 We now define the exterior differential    ${\boldsymbol d}: \mathcal E^k(\mathbb R^n)   \longrightarrow  \mathcal E^{k+1} (\mathbb R^n)$   by
 $${\boldsymbol d}    =\sum_{m=1}^n {\boldsymbol \varepsilon}_{m} \partial_{m}       , $$
 and the  co-differential  ${\boldsymbol   \delta} :    \mathcal E^{k+1}  (\mathbb R^n) \longrightarrow  \mathcal E^k(\mathbb R^n) $ by 
 $${\boldsymbol   \delta}  = - \sum_{m=1}^n   {\boldsymbol \iota}_{m}      \partial_{m}  ,$$
  where $\partial_{m }$ is the directional derivative in the direction of the basis vector $e_m.$ We set ${\boldsymbol   \delta} =0$ on $\mathcal E^0(\mathbb R^n)=C^\infty(\mathbb R^n).$ In the light of  \autoref{commi}, \autoref{comm} and   \autoref{ID},  direct computations show ${\boldsymbol d} \circ {\boldsymbol d}={\boldsymbol   \delta} \circ {\boldsymbol   \delta}=0 $  
  and 
  \begin{equation}\label{tikk1}  
 {\boldsymbol d}  {\boldsymbol \iota}_{j} + {\boldsymbol \iota}_{j}  {\boldsymbol d} =\partial_j,\qquad    {\boldsymbol \delta}  {\boldsymbol \varepsilon}_{j} + {\boldsymbol \varepsilon}_{j}  {\boldsymbol \delta} =-\partial_j,\qquad 0\leq j\leq n.
  \end{equation}    

We close this section by introducing the   Hodge Laplacian on differential  forms, defined by 
\begin{equation}\label{hodge} \Delta:= - ({\boldsymbol d} {\boldsymbol \delta} +{\boldsymbol \delta}{\boldsymbol d}) = \sum_{m=1}^n \partial_m^2.\end{equation}
 Henceforth we will denote the Hodge Laplacian  by $Q\big({\partial \over {\partial x}}\big)$ where 
 \begin{equation}\label{qua} 
 Q(x):=x_1^2+\cdots+x_n^2.
 \end{equation}

   \section{The main identities for Riesz distributions on differential forms}   
For ${s}$ a complex parameter, consider for    $\varphi\in \mathcal S(\mathbb  R^n)$  the formula \begin{equation}\label{cla}  \big\langle\mathcal R_{s} ,\varphi\big\rangle=   \pi^{-{n\over 2} }2^{-n-s} { {\Gamma\left(-{{s}\over 2}\right)} \over  {\Gamma\left( {{n+{s}}\over 2}\right)}}\int_{\mathbb R^n} \varphi(x) \Vert x\Vert^{s} dx .
 \end{equation}
For  $-n<\Rel {s} <0$, this formula defines a tempered distribution depending holomorphically on $s.$ The normalization factor is chosen for convenience, as we shall see below (see \autoref{friez}). By standard argument (see \cite{gels}) it can be analytically continued to $\mathbb C$, yielding a meromorphic family of tempered distributions, called the \emph{Riesz distributions}.
  
We follow the following convention for the Fourier transform on Schwartz functions $\varphi \in \mathcal S(\mathbb R^n)$:
$$\mathcal F(\varphi)(\xi)=\int_{\mathbb R^n} \varphi(x) e^{i\langle x,\xi\rangle}dx,$$ which extends to the space of tempered distributions $\mathcal S'(\mathbb R^n).$ It is known (see \cite{gels} or \cite[p. 38]{stri}) that the image of $\mathcal R_{s}$ by the Fourier transform is 
\begin{equation}\label{friez}
 \mathcal F(\mathcal R_{s})(\xi) = \Vert \xi\Vert^{-n-{s}}. \end{equation}

The classical Riesz distributions offer a good motivation for defining Riesz distributions for differential forms on $\mathbb R^n$ with coefficients in the Schwartz space $\mathcal S(\mathbb R^n).$ 

For  $0\leq k\leq n$, let $ \mathcal S\mathcal E^k(\mathbb R^n)   $ (resp. $ \mathcal S'\mathcal E^k(\mathbb R^n)   $) be the space of differential $k$-forms represented as in \autoref{repkfo} with coefficients 
$\omega_{i_1,   \ldots, i_k}$ in $\mathcal S(\mathbb R^n)$ (resp. $\mathcal S'(\mathbb R^n)$ ). We pin down that we may extend the Fourier  transform $\mathcal F$   on the space $   \mathcal S\mathcal E^k(\mathbb R^n)   $ by acting on the coefficients $\omega_{i_1,   \ldots, i_k}$  of the form  \autoref{repkfo}. 
  
For $0\leq k\leq n$ and a complex parameter  $s$,  let 
  $ \mathcal R_{s}^k$ be the \emph{Riesz distribution on differential forms} defined by

  \begin{equation}\label{riezf2} 
  \langle \mathcal R_{s}^k,  \omega\rangle    =   \pi^{-{n\over 2}}  2^{-{s}-n+1} {{\Gamma\left(-{{s}\over 2}+1\right)}\over {\Gamma\left({{{s}+n}\over 2}\right)}}      \int_{\mathbb R^n}    \Vert x\Vert ^{s-2} ({\boldsymbol \iota}_{x} {\boldsymbol \varepsilon}_{x} -{\boldsymbol \varepsilon}_{x} {\boldsymbol \iota}_{x}) \omega(x) dx, 
  \end{equation} 
 with  $\omega\in  \mathcal S\mathcal E^k (\mathbb R^n).$ For $-n<\Rel {s} <0$, this formula defines a tempered distribution depending holomorphically on $s.$ By \cite[\S 3.2]{fo},    $  \mathcal R_{s}^k$ can be analytically  continued to $\mathbb C$, giving a meromorphic family of tempered distributions.  When  $k=0,$ the identity 
  ${\boldsymbol \iota}_{x} {\boldsymbol \varepsilon}_{x} +{\boldsymbol \varepsilon}_{x} {\boldsymbol \iota}_{x}=\Vert x\Vert^ 2\Id_\mathcal E  $ implies immediately that $\mathcal R_{s}^0$ is nothing but the   classical Riesz distribution \autoref{cla}.        
       

In    \cite[Theorem 3.2]{fo} the authors proved that the Fourier transform of  $\mathcal R_{s}^k$ is given by 
\begin{equation}\label{fo2}
\mathcal F(  \mathcal R_{s}^k)(\xi)=   \Vert \xi\Vert^{-{s}-n-2} \Bigl(-({s}+2k) {\boldsymbol \iota}_{\xi} {\boldsymbol \varepsilon}_{\xi}+(s+2n-2k) 
{\boldsymbol \varepsilon}_{\xi} {\boldsymbol \iota}_{\xi}\Bigr)
\end{equation}
for $s$ not a pole. 
 Due to the facts ${\boldsymbol \iota}_{\xi}=0$ and ${\boldsymbol \iota}_{\xi} {\boldsymbol \varepsilon}_{\xi}=\Vert \xi\Vert^2\Id_\mathcal E$ on $\mathcal S\mathcal E^0 (\mathbb R^n) =\mathcal S (\mathbb R^n),$ it follows that \autoref{fo2} for  $k=0$ coincides with  \autoref{friez}.
  
In order to simplify notation, it is convenient to let 
 \begin{equation}\label{Z}
\mathcal Z_s^k := \mathcal F(\mathcal R_{-{s}-n}^k)
\end{equation} i.e.
$$
\mathcal Z_s^k (x) =      \Vert x\Vert^{{s}-2}    \Big(({s}+n-2k) {\boldsymbol \iota}_{x} {\boldsymbol \varepsilon}_{x}-(s-n+2k) 
  {\boldsymbol \varepsilon}_{x} {\boldsymbol \iota}_{x}\Big).$$


  We should point that in all arguments  below we first prove the desired result when the complex parameter  $s$  is so that everything makes sense, 
  and   then  we extend it   meromorphically        to the   complex plane $\mathbb C.$ 
  
 We shall need the following crucial result (which might be  of some interest in its own right):
   \begin{theorem}\label{shif}
  The distribution  $\mathcal Z_{s}^k$ satisfies the following properties: 
\begin{eqnarray} 
 \mathcal Z_{s}^k (x) &=&  \mathcal Z_{{s}-2}^k (x) \Big(  a_{k,s}  {\boldsymbol \iota}_{x} {\boldsymbol \varepsilon}_{x}+   b_{k,s}    {\boldsymbol \varepsilon}_{x} {\boldsymbol \iota}_{x}\Big), \label{I}\\
 {\partial \over  {\partial {x_j}}} \, \mathcal Z_{s}^k (x) &=&    \mathcal Z_{{s}-2}^k (x) \Big(s x_j \Id +   c_{k,{s}}  {\boldsymbol \iota}_{x} {\boldsymbol \varepsilon}_{j}+    d_{k,{s}}    {\boldsymbol \varepsilon}_{x} {\boldsymbol \iota}_{j}\Big),\label{II}\\
Q \Big( {\partial \over  {\partial {x}}} \Big)\mathcal Z_{s}^k (x) &=&{s}({s}+n) \mathcal Z_{{s}-2}^k (x),\label{III}
 \end{eqnarray}
 with \begin{equation}\label{alphabeta}   a_{k,s}  =   {{{s}+n-2k}\over { {s}+n-2k-2}},\qquad   b_{k,s} =  {{{s}-n+2k}\over {{s}-n+2k-2}} ,
 \end{equation}
 and  \begin{equation}\label{gammadelta}   c_{k,{s}}  =   {{2{s}}\over { {s}+n-2k-2}},\qquad   d_{k,{s}} =  {{2{s}}\over {{s}-n+2k-2}}.\end{equation}
 Above $Q \big( {\partial \over  {\partial {x}}} \big)$ denotes the Hodge Laplacian (see \autoref{hodge}).
  \end{theorem}
  \begin{proof}
 (1) On one hand,  we may rewrite   $\mathcal Z_s^k (x)$  as 
  $$\mathcal Z_{s}^k (x) = 
 \Vert x\Vert^2 \mathcal Z_{{s}-2}^k(x) -2{s}({s}+n)  \Vert x\Vert^{{s}-2} ( {\boldsymbol \iota}_{x} {\boldsymbol \varepsilon}_{x}-
  {\boldsymbol \varepsilon}_{x}{\boldsymbol \iota}_{x}) .$$
 On the other hand,  using  the fact ${\boldsymbol \iota}_{x} {\boldsymbol \varepsilon}_{x}+
  {\boldsymbol \varepsilon}_{x}{\boldsymbol \iota}_{x}=\Vert x\Vert^2\Id_{\mathcal E}$, we may rewrite the term $\Vert x\Vert^{2} ( {\boldsymbol \iota}_{x} {\boldsymbol \varepsilon}_{x}-
  {\boldsymbol \varepsilon}_{x}{\boldsymbol \iota}_{x}) $ as follows: 
\begin{equation}\label{trik}
  \Vert x\Vert^{2} ( {\boldsymbol \iota}_{x} {\boldsymbol \varepsilon}_{x}-
  {\boldsymbol \varepsilon}_{x}{\boldsymbol \iota}_{x})= \Big( (s+n-2k-2) {\boldsymbol \iota}_{x} {\boldsymbol \varepsilon}_{x} -(s-n+2k-2)  {\boldsymbol \varepsilon}_{x}{\boldsymbol \iota}_{x}\Big) \Big(a'_{k,s} {\boldsymbol \iota}_{x} {\boldsymbol \varepsilon}_{x} +b'_{k,s} {\boldsymbol \varepsilon}_{x}{\boldsymbol \iota}_{x}\Big),\end{equation}
  where 
  \begin{equation*}
  a'_{k,{s}}= {1\over{{s}+n-2k-2}},\qquad b'_{k,{s}}={1\over {{s}-n+2k-2}}.
  \end{equation*}
  Thus, 
  $$\mathcal Z_{s}^k (x) =   \Vert x\Vert^2 \mathcal Z_{{s}-2}^k (x) +2  \mathcal Z_{{s}-2}^k (x)  \ (a'_{k,s} 
  {\boldsymbol \iota}_{x} {\boldsymbol \varepsilon}_{x} +b'_{k,s} {\boldsymbol \varepsilon}_{x}{\boldsymbol \iota}_{x}) . $$
  Using again the identity  ${\boldsymbol \iota}_{x} {\boldsymbol \varepsilon}_{x}+
  {\boldsymbol \varepsilon}_{x}{\boldsymbol \iota}_{x}=\Vert x\Vert^2\Id_{\mathcal E}$ to  deduce the first statement.
  
  (2) First we have 
  \begin{eqnarray*}
{\partial \over  {\partial {x_j}}}({\boldsymbol \iota}_{x} {\boldsymbol \varepsilon}_{x}) &=&{\boldsymbol \iota}_{j} {\boldsymbol \varepsilon}_{x} +{\boldsymbol \iota}_{x} {\boldsymbol \varepsilon}_{j} = x_j \Id_{\mathcal E}-  {\boldsymbol \varepsilon}_{x}  {\boldsymbol \iota}_{j} +  {\boldsymbol \iota}_{x}  {\boldsymbol \varepsilon}_{j},\label{a}\\
  {\partial \over  {\partial {x_j}}} ( {\boldsymbol \varepsilon}_{x} {\boldsymbol \iota}_{x}) &=& {\boldsymbol \varepsilon}_{j} {\boldsymbol \iota}_{x}  + {\boldsymbol \varepsilon}_{x}{\boldsymbol \iota}_{j}  = x_j \Id_{\mathcal E}+  {\boldsymbol \varepsilon}_{x}  {\boldsymbol \iota}_{j} -  {\boldsymbol \iota}_{x}  {\boldsymbol \varepsilon}_{j}, \label{b}
   \end{eqnarray*}
   where, for abbreviation, $ {\boldsymbol \iota}_{j}$ (resp. ${\boldsymbol \varepsilon}_{j}$)  denotes ${\boldsymbol \iota}_{e_j}$ (resp. ${\boldsymbol \varepsilon}_{e_j}$). Above we have used the identity   \autoref{ID}. Then 
   \begin{align*}
 & {\partial \over  {\partial {x_j}}} \mathcal Z_{s}^k (x) 
   = 
    ({s}-2) x_j \Vert x\Vert^{{s}-4} \Big(({s}+n-2k)  {\boldsymbol \iota}_{x} {\boldsymbol \varepsilon}_{x}  -({s}-n+2k) {\boldsymbol \varepsilon}_{x}{\boldsymbol \iota}_{x} \Big) \\
& +  \Vert x\Vert^{{s}-2} \Big(  ({s}+n-2k)  ( x_j \Id_{\mathcal E}-  {\boldsymbol \varepsilon}_{x}  {\boldsymbol \iota}_{j} +  {\boldsymbol \iota}_{x}  {\boldsymbol \varepsilon}_{j})  -({s}-n+2k)  (x_j \Id_{\mathcal E}+  {\boldsymbol \varepsilon}_{x}  {\boldsymbol \iota}_{j} -  {\boldsymbol \iota}_{x}  {\boldsymbol \varepsilon}_{j} )\Big)\\
&  =     ({s}-2) x_j \Vert x\Vert^{{s}-4} \Big( 
\big( ({s}-2+n-2k)  {\boldsymbol \iota}_{x} {\boldsymbol \varepsilon}_{x}   -({s}-2-n+2k) {\boldsymbol \varepsilon}_{x}{\boldsymbol \iota _{x}\big) +2
( {\boldsymbol \iota}_{x} {\boldsymbol \varepsilon}_{x} - \boldsymbol \varepsilon}_{x}{\boldsymbol \iota}_{x} ) \Big)\\
 &   +  \Vert x\Vert^{{s}-2} \Big( 2 ( n-2k)x_j  \Id_{\mathcal E} + 2{s}  
 (  {\boldsymbol \iota}_{x}  {\boldsymbol \varepsilon}_{j}  - {\boldsymbol \varepsilon}_{x} {\boldsymbol \iota}_{j}   )\Big)\\
&  =   ({s}-2) x_j \mathcal Z_{{s}-2}^k (x)+    2 x_j \Vert x\Vert^{{s}-4} \Big(
(n-2k) \Vert x\Vert^2 \Id_{\mathcal E}+({s}-2) (  {\boldsymbol \iota}_{x}  {\boldsymbol \varepsilon}_{j}  - {\boldsymbol \varepsilon}_{x} {\boldsymbol \iota}_{j}   )\Big)\\
& + 
   2 {s} \Vert x\Vert^{{s}-2} (  {\boldsymbol \iota}_{x}  {\boldsymbol \varepsilon}_{j}  - {\boldsymbol \varepsilon}_{x} {\boldsymbol \iota}_{j}   )\\
&   =    ({s}-2) x_j \mathcal Z_{{s}-2}^k (x) +{2 } x_j \mathcal Z_{{s}-2}^k (x) +2{s}   \Vert x\Vert^{{s}-2} 
  (  {\boldsymbol \iota}_{x}  {\boldsymbol \varepsilon}_{j}  - {\boldsymbol \varepsilon}_{x} {\boldsymbol \iota}_{j}   ). 
  \end{align*}
Now using again the   trick    \autoref{trik}   we obtain 
 $$ {\partial \over  {\partial {x_j}}} \mathcal Z_{s}^k (x)  =  {s}    x_j \mathcal Z_{{s}-2}^k (x)  +2{s}    \mathcal Z_{{s}-2}^k (x)
     ( a'_{k,{s}}{\boldsymbol \iota}_{x}  {\boldsymbol \varepsilon}_{j}  +b'_{k,{s}}{\boldsymbol \varepsilon}_{x} {\boldsymbol \iota}_{j}   ),$$
  and \autoref{II} follows.
     
(3) The definition  of   $\mathcal Z_{s}^k (x)$ and the    intertwining property $ Q\big ( {\partial \over  {\partial {x }}}\big) \circ \mathcal F= -\mathcal F\circ \Vert \cdot\Vert^2 $ imply 
\begin{equation*}
Q\Big ( {\partial \over  {\partial {x }}}\Big)  \,\mathcal Z_{s}^k (x) =- \mathcal F( \Vert \cdot \Vert ^2 \mathcal R_{-{s}-n}^k)(x)
={s}({s}+n) \mathcal F(  \mathcal R_{-{s}-n+2}^k)(x) ={s}({s}+n) \mathcal Z_{{s}-2}^k (x).
\end{equation*}
  \end{proof}
  
  For $0\leq k,\ell\leq n,$ we define the  space $\mathcal E^{k,\ell}(\mathbb R^n\times \mathbb R^n) $ as the space of smooth functions on $\mathbb R^n\times \mathbb R^n$ with values in $\Lambda^k \otimes \Lambda^\ell.$  
  More generally, denote by 
  $\mathcal D\mathcal E^{k,\ell }(\mathbb R^n\times \mathbb R^n)$ (resp.  $\mathcal S\mathcal E^{k,\ell}(\mathbb R^n\times \mathbb R^n)$, $\mathcal S'\mathcal E^{ k,\ell}  (\mathbb R^n\times\mathbb R^n)$) the space of differential forms of bidegree $(k,\ell)$ on $\mathbb R^n\times \mathbb R^n$ with  coefficients in $\mathcal D(\mathbb R^n\times \mathbb R^n)$ (resp. $\mathcal S(\mathbb R^n\times \mathbb R^n)$, $\mathcal S'(\mathbb R^n\times\mathbb R^n)$)


  \newpage
  
  \begin{theorem}\label{main1}   
For every  $\omega \in  \mathcal S\mathcal E^{k,\ell} (\mathbb R^n\times\mathbb R^n) $, the following formula hols true 
  $$Q\left( {{\partial}\over {\partial x}} - {{\partial}\over {\partial y}}\right)  \left( \mathcal Z_{s}^k(x)\otimes \mathcal Z_{t}^\ell(y) \; \omega(x,y) \right) = \mathcal Z_{{s}-2}^k(x) \otimes \mathcal Z _{{t}-2}^\ell(y) \,   D_{{s},{t}}^{k,\ell} \; \omega(x,y),   $$  where   $ D_{s,t}^{k,\ell}$ is the differential operator on $(k,\ell)$-differential forms  given by  
 
\begin{eqnarray*}
  D_{s,t}^{k,\ell}&=&    \Big(  a_{k,s}  {\boldsymbol \iota}_{x}  {\boldsymbol \varepsilon}_{x}  + b_{k,s}  {\boldsymbol \varepsilon}_{x} {\boldsymbol \iota}_{x}\Big) \otimes \Big(  a_{\ell,t}  {\boldsymbol \iota}_{y}  {\boldsymbol \varepsilon}_{y} + b_{\ell,t}  {\boldsymbol \varepsilon}_{y} {\boldsymbol \iota}_{y}\Big)  \circ Q \left({{\partial}\over {\partial x}}- {{\partial}\over {\partial y}}\right)\\
  &&+ 2  \sum_{j=1}^n      
  \Big(sx_j\Id_{\mathcal E^k} + c_{k,s} {\boldsymbol \iota}_{x}  {\boldsymbol \varepsilon}_{j}  +d_{k,s} {\boldsymbol \varepsilon}_{x} {\boldsymbol \iota}_{j}\Big)\otimes 
  \Big(   a_{\ell,t}  {\boldsymbol \iota}_{y}  {\boldsymbol \varepsilon}_{y} + b_{\ell,t}  {\boldsymbol \varepsilon}_{y} {\boldsymbol \iota}_{y}\Big) \circ  \Big({{\partial}\over {\partial x_j}}  -  {{\partial}\over {\partial y_j}}  \Big)  \\
  &&- 2  \sum_{j=1}^n      
   \Big(   a_{k,s}  {\boldsymbol \iota}_{x}  {\boldsymbol \varepsilon}_{x} + b_{k,s}  {\boldsymbol \varepsilon}_{x} {\boldsymbol \iota}_{x}\Big)
  \otimes 
     \Big(ty_j\Id_{\mathcal E^\ell} + c_{\ell,t} {\boldsymbol \iota}_{y}  {\boldsymbol \varepsilon}_{j}  + d_{\ell,t} {\boldsymbol \varepsilon}_{y} {\boldsymbol \iota}_{j}\Big)\circ  \Big({{\partial}\over {\partial x_j}}  -  {{\partial}\over {\partial y_j}}  \Big) \\
     &&  -{{2 } }   \sum_{j=1}^n \Big(sx_j\Id_{\mathcal E^k} + c_{k,s} {\boldsymbol \iota}_{x}  {\boldsymbol \varepsilon}_{j}  +d_{k,s} {\boldsymbol \varepsilon}_{x} {\boldsymbol \iota}_{j}\Big) \otimes \Big(ty_j\Id_{\mathcal E^\ell} + c_{\ell,t} {\boldsymbol \iota}_{y}  {\boldsymbol \varepsilon}_{j} + d_{\ell,t} {\boldsymbol \varepsilon}_{y} {\boldsymbol \iota}_{j}\Big)  \\
 &&+    {s}({s}+n)  \Id_{\mathcal E^k} \otimes \Big(  a_{\ell,t}  {\boldsymbol \iota}_{y}  {\boldsymbol \varepsilon}_{y}+ b_{\ell,t} {\boldsymbol \varepsilon}_{y} {\boldsymbol \iota}_{y} \Big) 
 +  {t}({t}+n)   \Big(   a_{k,s}  {\boldsymbol \iota}_{x}  {\boldsymbol \varepsilon}_{x}+  b_{k,s} {\boldsymbol \varepsilon}_{x} {\boldsymbol \iota}_{x} \Big) \otimes \Id_{\mathcal E^\ell} .
 \end{eqnarray*} 
The coefficients  $a_{k,{s}}$, $b_{k,{s}}$, $c_{k,{s}}$ and $d_{k,{s}}$  are given by  \autoref{alphabeta} and \autoref{gammadelta} (and similarly when the subscripts $k,{s}$ are  replaced by   $\ell,{t}$). 
  \end{theorem} 
  \begin{proof}
A routine  calculation gives  
\begin{align*}
&Q\left( {{\partial}\over {\partial x}} - {{\partial}\over {\partial y}}\right)  \Big( \mathcal Z_{s }^k(x)\otimes \mathcal Z _{t }^\ell(y)  \omega(x,y) \Big)   =  
\Big( Q\left( {{\partial}\over {\partial x}} - {{\partial}\over {\partial y}}\right)    \mathcal Z_{s }^k(x)\otimes \mathcal Z _{t }^\ell(y) \Big)    \omega(x,y) \\
 &+ \mathcal Z_{s }^k(x) \otimes \mathcal Z _{t }^\ell(y) Q \left( {{\partial}\over {\partial x}} - {{\partial}\over {\partial y}}\right)   \omega(x,y)  
 + \widetilde Q\left( {{\partial}\over {\partial x}} , {{\partial}\over {\partial y}}\right) \Big(  \mathcal Z_{s }^k(x)\otimes \mathcal Z _{t }^\ell(y)  \omega(x,y) \Big),  
\end{align*}
where 
  \begin{align*}
\lefteqn{
\widetilde Q\left( {{\partial}\over {\partial x}} , {{\partial}\over {\partial y}}\right) \Big(  \mathcal Z_{s }^k(x)\otimes \mathcal Z _{t }^\ell(y)  \omega(x,y) \Big)}\\
  =&     2 \sum_{j=1}^n\Big(  {{\partial}\over {\partial x_j}}  \mathcal Z_s^k(x)\otimes \mathcal Z_t^\ell (y)  {{\partial}\over {\partial x_j}}  \omega(x,y) +   \mathcal Z_s^k(x)\otimes  {{\partial}\over {\partial y_j}}  \mathcal Z_t^\ell (y) {{\partial}\over {\partial y_j}} \omega(x,y)\Big) \\
&  -2 \sum_{j=1}^n  \Big(\mathcal Z_s^k(x)\otimes {{\partial}\over {\partial y_j}} \mathcal Z_t^\ell (y) {{\partial}\over {\partial x_j}}  \omega(x,y) 
   -    {{\partial}\over {\partial x_j}} \mathcal Z_s^k(x)\otimes \mathcal Z_t^\ell (y) {{\partial}\over {\partial y_j}} \omega(x,y)\Big).
\end{align*}  
Firstly, in view of the identities  \autoref{II} and \autoref{III},  we have
 
    \begin{align*}
&\Big(Q\left( {{\partial}\over {\partial x}} - {{\partial}\over {\partial y}}\right)    \mathcal Z_{s }^k(x)\otimes \mathcal Z _{t }^\ell(y)  \Big)   \omega(x,y) \\
=&   Q \left( {{\partial}\over {\partial x}}\right) \mathcal Z_{s }^k(x)\otimes \mathcal Z_t^\ell (y)   \omega(x,y)+
   \mathcal Z_{s }^k(x)\otimes Q \left( {{\partial}\over {\partial y}}\right)  \mathcal Z_t^\ell (y)   \omega(x,y)\\
 & -2 \sum_{j=1}^n {{\partial}\over {\partial x_j}} \mathcal Z_s^k(x)\otimes   {{\partial}\over {\partial y_j}} \mathcal Z_t^\ell (y)   \omega(x,y)\\
    =&
s(s+n)    \mathcal Z_{s-2}^k(x)\otimes \mathcal Z _{t}^\ell(y)   \omega(x,y)
+t(t+n) \mathcal Z_{s }^k(x)\otimes \mathcal Z _{t-2}^\ell(y)   \omega(x,y) \\
&  -2 
 \mathcal Z_{s-2}^k(x)\otimes \mathcal Z _{t-2}^\ell(y)  \\
&  \Big\{\sum_{j=1}^n \Big(sx_j\Id_{\mathcal E^k} +c_{k,s} {\boldsymbol \iota}_{x}  {\boldsymbol \varepsilon}_{j}  +d_{k,s} {\boldsymbol \varepsilon}_{x} {\boldsymbol \iota}_{j}\Big)\otimes \Big(ty_j\Id_{\mathcal E^\ell} +c_{\ell,t} {\boldsymbol \iota}_{y}  {\boldsymbol \varepsilon}_{j} +d_{\ell,t} {\boldsymbol \varepsilon}_{y} {\boldsymbol \iota}_{j}\Big)\Big\} \omega(x,y)\\
    =&\mathcal Z_{s-2}^k(x)\otimes \mathcal Z _{t-2}^\ell(y) \\
 & \Big\{
 s(s+n) \Id_{\mathcal E^k} \otimes \Big(  a_{\ell,t}  {\boldsymbol \iota}_{y}  {\boldsymbol \varepsilon}_{y}+ b_{\ell,t} {\boldsymbol \varepsilon}_{y} {\boldsymbol \iota}_{y} \Big )
 +
t(t+n)  \Big(  a_{k,s}  {\boldsymbol \iota}_{x}  {\boldsymbol \varepsilon}_{x}+ b_{k,s} {\boldsymbol \varepsilon}_{x} {\boldsymbol \iota}_{x} \Big)\otimes \Id_{\mathcal E^\ell}
  \\
&   -{{2} } \sum_{j=1}^n \Big(sx_j\Id_{\mathcal E^k} + c_{k,s} {\boldsymbol \iota}_{x}  {\boldsymbol \varepsilon}_{j}  + d_{k,s} {\boldsymbol \varepsilon}_{x} {\boldsymbol \iota}_{j}\Big)\otimes \Big(ty_j\Id_{\mathcal E^\ell} + c_{\ell,t} {\boldsymbol \iota}_{y}  {\boldsymbol \varepsilon}_{j} + d_{\ell,t} {\boldsymbol \varepsilon}_{y} {\boldsymbol \iota}_{j}\Big)\Big\} \omega(x,y).
  \end{align*} 
 
Secondly, the identity  \autoref{I} gives 
 
\begin{align*}
 &\mathcal Z_{s }^k(x)\otimes \mathcal Z _{t }^\ell(y) Q \left( {{\partial}\over {\partial x}} - {{\partial}\over {\partial y}}\right)   \omega(x,y)  \\
&=  \mathcal Z_{s-2}^k(x)\otimes \mathcal Z _{t-2}^\ell(y) 
 \Big\{
   \Big(  a_{k,s}  {\boldsymbol \iota}_{x}  {\boldsymbol \varepsilon}_{x}  + b_{k,s}  {\boldsymbol \varepsilon}_{x} {\boldsymbol \iota}_{x}\Big)\otimes \Big(  a_{\ell,t}  {\boldsymbol \iota}_{y}  {\boldsymbol \varepsilon}_{y} + b_{\ell,t}  {\boldsymbol \varepsilon}_{y} {\boldsymbol \iota}_{y}\Big)  
Q\left({{\partial}\over {\partial x}} - {{\partial}\over {\partial y}}\right)\Big\}\omega(x,y).
\end{align*}
 
Finally, using again  \autoref{I}  and \autoref{II}   we  obtain 
\begin{align*}
&\widetilde Q\left( {{\partial}\over {\partial x}} , {{\partial}\over {\partial y}}\right) \left(  \mathcal Z_{s }^k(x)\otimes \mathcal Z _{t }^\ell(y)  \omega(x,y) \right)
  = 2\mathcal Z_{s-2}^k(x)\otimes \mathcal Z _{t-2}^\ell(y)\Big\{\\
  & \sum_{j=1}^n      \Big(sx_j\Id_{\mathcal E^k} + c_{k,s} {\boldsymbol \iota}_{x}  {\boldsymbol \varepsilon}_{j}  + d_{k,s} {\boldsymbol \varepsilon}_{x} {\boldsymbol \iota}_{j}\Big)\otimes 
  \Big(  a_{\ell,t}  {\boldsymbol \iota}_{y}  {\boldsymbol \varepsilon}_{y} + b_{\ell,t}  {\boldsymbol \varepsilon}_{y} {\boldsymbol \iota}_{y}\Big) {{\partial \omega}\over{\partial {x_j}}} (x,y)\\
  & +\sum_{j=1}^n      \Big( a_{k,s} {\boldsymbol \iota}_{x}  {\boldsymbol \varepsilon}_{x} + b_{k,s}  {\boldsymbol \varepsilon}_{x} {\boldsymbol \iota}_{x}\Big)
  \otimes 
     \Big(ty_j\Id_{\mathcal E^\ell} + c_{\ell,t} {\boldsymbol \iota}_{y}  {\boldsymbol \varepsilon}_{j}  + d_{\ell,t} {\boldsymbol \varepsilon}_{y} {\boldsymbol \iota}_{j}\Big)
  {{\partial \omega}\over{\partial {y_j}}}(x,y)\\
  &- \sum_{j=1}^n   
   \Big(  a_{k,s}  {\boldsymbol \iota}_{x}  {\boldsymbol \varepsilon}_{x} + b_{k,s}  {\boldsymbol \varepsilon}_{x} {\boldsymbol \iota}_{x}\Big)
  \otimes 
     \Big(ty_j\Id_{\mathcal E^\ell} + c_{\ell,t} {\boldsymbol \iota}_{y}  {\boldsymbol \varepsilon}_{j}  + d_{\ell,t} {\boldsymbol \varepsilon}_{y} {\boldsymbol \iota}_{j}\Big)
{{\partial \omega}\over{\partial {x_j}}}(x,y)\\
&-  \sum_{j=1}^n   \Big(sx_j\Id_{\mathcal E^k} + c_{k,s} {\boldsymbol \iota}_{x}  {\boldsymbol \varepsilon}_{j}  + d_{k,s} {\boldsymbol \varepsilon}_{x} {\boldsymbol \iota}_{j}\Big)\otimes 
  \Big(  a_{\ell,t}  {\boldsymbol \iota}_{y}  {\boldsymbol \varepsilon}_{y} + b_{\ell,t}  {\boldsymbol \varepsilon}_{y} {\boldsymbol \iota}_{y}\Big) {{\partial \omega}\over{\partial {y_j}}} (x,y)\Big\}\\
  =& 2\mathcal Z_{s-2}^k(x)\otimes \mathcal Z _{t-2}^\ell(y)\Big\{\\
  &  \sum_{j=1}^n    
  \Big(sx_j\Id_{\mathcal E^k} + c_{k,s} {\boldsymbol \iota}_{x}  {\boldsymbol \varepsilon}_{j}  + d_{k,s} {\boldsymbol \varepsilon}_{x} {\boldsymbol \iota}_{j}\Big)\otimes 
  \Big(  a_{\ell,t}  {\boldsymbol \iota}_{y}  {\boldsymbol \varepsilon}_{y} + b_{\ell,t} {\boldsymbol \varepsilon}_{y} {\boldsymbol \iota}_{y}\Big)  \Big({{\partial}\over{\partial {x_j}}}-{{\partial}\over{\partial {y_j}}}\Big)\omega(x,y) \\
  &-\sum_{j=1}^n      \Big(  a_{k,s}  {\boldsymbol \iota}_{x}  {\boldsymbol \varepsilon}_{x} + b_{k,s}  {\boldsymbol \varepsilon}_{x} {\boldsymbol \iota}_{x}\Big)
  \otimes 
     \Big(ty_j\Id_{\mathcal E^\ell} + c_{\ell,t} {\boldsymbol \iota}_{y}  {\boldsymbol \varepsilon}_{j}  + d_{\ell,t} {\boldsymbol \varepsilon}_{y} {\boldsymbol \iota}_{j}\Big)  \Big({{\partial}\over{\partial {x_j}}}- {{\partial}\over{\partial {y_j}}}\Big)\omega(x,y)\Big\}. \\
    \end{align*}
    This finishes the proof of the theorem.
    \end{proof}
 
  For $s\in\mathbb{C}$, let
 \begin{equation}\label{conv-op}
 J_s^k\omega(x):=\int_{\mathbb R^n}\mathcal R_s^k(x-y)\omega(y)dy,\quad \omega\in \mathcal S\mathcal E^k(\mathbb R^n).
 \end{equation}
 We may see $J_s^k$ as a convolution operator with the distribution $\mathcal R_s^k$,
  $$J_s^k \omega = \mathcal R_s^k\ast \omega.$$
 So \autoref{conv-op}  defines a meromorphic family of operators from $\mathcal S\mathcal E^k(\mathbb R^n)$ to $\mathcal S'\mathcal E^k(\mathbb R^n)$.

 Recall the following formulas for the Fourier transform :


\begin{alignat}{5}
\mathcal F\Big({{\partial \omega}\over{\partial {y_j}}} \Big) (x)&=& -\sqrt{-1}\, x_j \mathcal F(\omega)(x),  &\qquad &  \mathcal F(y_j \omega)(x) &=&- \sqrt{-1}\, {{\partial}\over{\partial {x_j}}} \mathcal F(\omega)(x) \label{TF1}\\
\mathcal F(\boldsymbol d\omega) (x) &=& -\sqrt{-1}\, \boldsymbol \varepsilon_x \mathcal F \omega (x),  &\qquad & \mathcal F ( \boldsymbol \varepsilon_y \omega)(x) &=&-\sqrt {-1} \, \boldsymbol d \mathcal F(\omega) (x)  \label{TF2}\\
\mathcal F (\boldsymbol \delta \omega) (x) &=&\sqrt{-1} \,\boldsymbol \iota_x \mathcal F(\omega)(x),  &\qquad &  \mathcal F ( \boldsymbol \iota_y \omega)(x) &=&\sqrt{-1}  \,\boldsymbol \delta  \mathcal F(\omega) (x). \label{TF3}
\end{alignat}

 For $s\in\mathbb{C}$, let

\begin{alignat}{5}
\alpha_{k,s}&=(s+n-2k)(s-n+2k-2)& &\qquad& \beta_{k,s}&=(s-n+2k)(s+n-2k-2)& \\
   \gamma_{k,s}&=2s(s-n+2k-2)& &\qquad& \delta_{k,s} &=2s(s+n-2k-2)& 
\end{alignat}
and
\begin{equation}
\kappa_{k,s}=(s-n+2k-2)(s+n-2k-2)
\end{equation}
 
\begin{theorem}    \label{main2} The following identity holds true
  $$ - \kappa_{k,s} \kappa_{\ell,t}\,\Vert x-y\Vert^2  \big(  J_{-s-n }^k\otimes J_{-t-n}^\ell \big)= \big( J_{-s-n+2}^k \otimes   J_{  -t-n+2}^\ell \big) \circ   E_{s,t}^{k,\ell}  ,$$ where $ E_{s,t}^{k,\ell}$ is the differential operator with polynomial coefficients in $x, y$ (and also in $s, t$) defined on $(k,\ell)$-differential forms by 
\begin{eqnarray}\label{estkl}
  E_{s,t}^{k,\ell}&=&  
    - \Big( \alpha_{k,s}  {\boldsymbol \delta}   {\boldsymbol d}   + \beta_{k,s}   {\boldsymbol d}  {\boldsymbol \delta} \Big) \otimes \Big(  \alpha_{\ell,t}  {\boldsymbol \delta}   {\boldsymbol d}  + \beta_{\ell,t}  {\boldsymbol d}  {\boldsymbol \delta}\Big)  \circ \Vert x-y\Vert^2 \nonumber\\
  &&-2\sum_{j=1}^n  
  \Big(s \kappa_{k,s} {\partial \over  {\partial {x_j }}}   - \gamma_{k,s} {\boldsymbol \delta}   {\boldsymbol \varepsilon}_{j}  +  \delta_{k,s} {\boldsymbol d}  {\boldsymbol \iota}_{j}\Big) \otimes 
  \Big( \alpha_{\ell,t}  {\boldsymbol \delta}   {\boldsymbol d}  + \beta_{\ell,t}  {\boldsymbol d}  {\boldsymbol \delta} \Big)  \circ (x_j -y_j) \nonumber \\
  &&+2\sum_{j=1}^n    
   \Big(  \alpha_{k,s}  {\boldsymbol \delta}  {\boldsymbol d}  + \beta_{k,s}  {\boldsymbol d}  {\boldsymbol \delta} \Big) 
  \otimes 
     \Big( t \kappa_{\ell,t}   {\partial \over  {\partial {y_j }}}   - \gamma_{\ell,t} {\boldsymbol \delta}   {\boldsymbol \varepsilon}_{j}  + \delta_{\ell,t} {\boldsymbol d}  {\boldsymbol \iota}_{j}\Big)  \circ (x_j - y_j )\nonumber \\
     &&  +{{2 } }   \sum_{j=1}^n \Big(s \kappa_{k,s}  {\partial \over  {\partial {x_j }}}   - \gamma_{k,s} {\boldsymbol \delta}   {\boldsymbol \varepsilon}_{j }  + \delta_{k,s} {\boldsymbol d}  {\boldsymbol \iota}_{j }\Big) \otimes \Big( t  \kappa_{\ell,t}  {\partial \over  {\partial {y_j }}}   - \gamma_{\ell,t} {\boldsymbol \delta}   {\boldsymbol \varepsilon}_{j } + \delta_{\ell,t} {\boldsymbol d}  {\boldsymbol \iota}_{j }\Big) \nonumber \\
  &&+   s(s+n) \kappa_{k,s} \Id_{\mathcal E^k}\kappa_{\ell,t} \otimes \Big(  \alpha_{\ell,t}  {\boldsymbol \delta}   {\boldsymbol d} + \beta_{\ell,t} {\boldsymbol d}  {\boldsymbol \delta} \Big)  +  t(t+n) \kappa_{\ell,t}  \Big(  \alpha_{k,s}  {\boldsymbol \delta}    {\boldsymbol d} + \beta_{k,s} {\boldsymbol d}  {\boldsymbol \delta} \Big) \otimes \Id_{\mathcal E^\ell} \nonumber\\
 \end{eqnarray} 
  \end{theorem} 
  
This is merely the Fourier transform version of  Theorem \ref{main1}, using \autoref{Z} and formulas   \autoref{TF1},  \autoref{TF2} and  \autoref{TF3}. We omit details.  \\

Next, we need to rewrite $  E_{s,t}^{k,\ell}$    in its  normal form (i.e. multiplications \emph{after} differentiations). Before stating the result,  let us introduce for $0\leq k\leq n$, $1\leq j \leq n$ the following differential operators:
\begin{eqnarray*}
 \Box_{k,s} \;\, &:= & \alpha_{k,s}  {\boldsymbol \delta}   {\boldsymbol d}   + \beta_{k,s}   {\boldsymbol d}  {\boldsymbol \delta},\\
 \nabla_{k,s,j} & := &   \Big(2\alpha_{k,s} {\partial\over {\partial x_j}}  -4(n-2k)    {\boldsymbol \varepsilon}_{j } {\boldsymbol \delta}  +4(n-2k)  {\boldsymbol d} {\boldsymbol \iota}_{j } \Big) - \Big(  s  \alpha_{k,s}  {\partial\over {\partial x_j}} + \gamma_{k,s}   {\boldsymbol \varepsilon}_{j} {\boldsymbol \delta}   + \delta_{k,s} {\boldsymbol d}  {\boldsymbol \iota}_{j}\Big)\\
 &=& (2-s) \alpha_{k,s}  {\partial\over {\partial x_j}} - ( 4(n-2k)  +\gamma_{k,s})   {\boldsymbol \varepsilon}_{j} {\boldsymbol \delta} + ( 4(n-2k)  -\delta_{k,s}) {\boldsymbol d}  {\boldsymbol \iota}_{j} \\
 &=&(2-s)\Big[(s+n-2k)(s-n+2k-2) {\partial\over {\partial x_j}}    +2(s-n+2k)  {\boldsymbol \varepsilon}_{j} {\boldsymbol \delta}  +2(s+n-2k) {\boldsymbol d}  {\boldsymbol \iota}_{j}\Big],
\end{eqnarray*}
where the coefficients $\alpha_{k,s},$ $\beta_{k,s}$, $\gamma_{k,s}$ and $\delta_{k,s}$  are given by  \autoref{alphabeta} and  \autoref{gammadelta}. Similarly, we introduce the operators $\boxvoid_{\ell,{t}}$  and 
 $ \nabla_{\ell,{t},j}$ with respect to the $y$-variable.

\begin{theorem}\label{normal} 
The operator $ E_{s,t}^{k,\ell}$ in Theorem \ref{main2} can be rewritten in the following normal form: 
\begin{eqnarray}\label{estkl2}
  E_{s,t}^{k,\ell}&=&    -\Vert x-y\Vert^2      \boxvoid_{k,s}\otimes \boxvoid_{\ell,t} \nonumber\\ 
  &&+ 2 \sum_{j=1}^n (x_j-y_j)   \left\{ \nabla_{k,s,j}   \otimes \boxvoid_{\ell,t}
  -   \boxvoid_{k,s} \otimes   \nabla_{\ell,t,j}  \right\} \nonumber\\
  &&+ 2 \sum_{j=1}^n    \nabla_{k,s,j}   \otimes   \nabla_{\ell,t,j}\nonumber \\
  &&+  (t-2)(t-n-2)(t-n+2\ell)(t+n-2\ell)     \boxvoid_{k,s} \otimes  \Id_{{\mathcal E}^\ell}\nonumber\\
&&  +  (s-2)(s-n-2)(s-n+2k)(s+n-2k)       \Id_{{\mathcal E}^k} \otimes \boxvoid_{\ell,t}.
 \end{eqnarray}

\end{theorem}

The proof is straightforward, but long and tedious. We first need some elementary formulas .

\begin{lemma}\label{ident_x}  Let $\omega$ be a $k$-form. Then, for fixed $y\in\mathbb{R}^n$,
\begin{eqnarray*}  
\cdd\dd  (x_j-y_j)\omega  &=& \displaystyle(x_j-y_j)\cdd \dd\omega   -2  {\partial \over  {\partial {x_j }}}\omega -  \de_j \cdd\omega +\dd \di_j\omega  \\
\dd\cdd  (x_j-y_j)\omega &=& \displaystyle (x_j-y_j)\dd \cdd\omega+ \de_j \cdd\omega   -\dd \di_j \omega\\
\cdd \dd \Vert x-y\Vert^2\omega  &=& \displaystyle  \Vert x-y\Vert^2 \cdd \dd \omega
+2\sum_{j=1}^n(x_j-y_j)( -2 {\partial \over  {\partial {x_j }}} -  \de_j  \cdd +\dd \di_j    )\omega  -2(n-k)\omega \\
\dd \cdd    \Vert x-y\Vert^2\omega &=&\displaystyle \Vert x-y\Vert^2 \dd \cdd\omega +2\sum_{j=1}^n(x_j-y_j)(\de_j \cdd   - \dd \di_j )\omega  -2k \omega.
\end{eqnarray*}
 Derivations are taken with respect to the $x$-variable.

\end{lemma}
 \begin{proof} This is a direct consequence of the identities
$$\begin{cases} 
\dd (x_j-y_j)\omega  =   \de_j\omega  +(x_j-y_j)\dd\omega \\
 \cdd (x_j-y_j)\omega  =   -\di_j\omega  +(x_j-y_j)\cdd\omega,
 \end{cases}
$$
  and 
  $$\begin{cases} 
\dd  \Vert x-y\Vert^2\omega  =  \Vert x-y\Vert^2\dd\omega+ 2\sum_{j=1}^n(x_j-y_j)\de_j\omega     \\
 \cdd  \Vert x-y\Vert^2\omega  =    \Vert x-y\Vert ^2\cdd\omega -2\sum_{j=1}^n(x_j-y_j)\di_j\omega.
 \end{cases}
$$
 
\end{proof}

We may now start the proof of Theorem \ref{normal}. In the light  of Lemma \ref{ident_x} (and its  version  with respect to $y$) together with the anti-commutator laws  in  \autoref{tikk1}   we arrive at the expressions below for the first three terms   of the operator 
$ E_{s,t}^{k,\ell}$ in Theorem \ref{main2}. The first  term  of the operator $  E_{s,t}^{k,\ell}$ in \autoref{estkl}   can be rewritten as: 
\begin{eqnarray*}
&& - \Vert x-y\Vert^2      \Big(  \alpha_{k,s}  {\boldsymbol \delta}   {\boldsymbol d}   + \beta_{k,s}   {\boldsymbol d}  {\boldsymbol \delta} \Big) \otimes \Big(   \alpha_{\ell,t}  {\boldsymbol \delta}   {\boldsymbol d}  + \beta_{\ell,t} {\boldsymbol d}  {\boldsymbol \delta}\Big)   \\
&&+4\sum_{j=1}^n (x_j-y_j) \Big(-2(n-2k)   ({\boldsymbol \varepsilon}_{j } {\boldsymbol \delta} -{\boldsymbol d} {\boldsymbol \iota}_{j } )+  \alpha_{k,s}   {\partial \over  {\partial {x_j }}}  \Big) \otimes 
\Big(  \alpha_{\ell,t}  {\boldsymbol \delta}   {\boldsymbol d}  + \beta_{\ell,t} {\boldsymbol d}  {\boldsymbol \delta}\Big) \\
&&-4 \sum_{j=1}^n (x_j-y_j) \Big(  \alpha_{k,s}  {\boldsymbol \delta}   {\boldsymbol d}  + \beta_{k,s} {\boldsymbol d}  {\boldsymbol \delta}\Big)  \otimes 
\Big(-2(n-2\ell)    ({\boldsymbol \varepsilon}_{j } {\boldsymbol \delta} -{\boldsymbol d} {\boldsymbol \iota}_{j } )+ \alpha_{\ell,t}   {\partial \over  {\partial {y_j }}}  \Big) \\
&&+ 8\sum_{j=1}^n \Big(-2(n-2k)   ({\boldsymbol \varepsilon}_{j } {\boldsymbol \delta} -{\boldsymbol d} {\boldsymbol \iota}_{j } )+  \alpha_{k,s}  {\partial \over  {\partial {x_j }}}  \Big) 
\otimes 
\Big(-2(n-2\ell)   ({\boldsymbol \varepsilon}_{j } {\boldsymbol \delta} -{\boldsymbol d} {\boldsymbol \iota}_{j } )+ \alpha_{\ell,t}   {\partial \over  {\partial {y_j }}}  \Big) \\
&&+2((n-\ell)  \alpha_{\ell,t}  +\ell \beta_{\ell,t} ) \Big(  \alpha_{k,s}  {\boldsymbol \delta}   {\boldsymbol d}  + \beta_{k,s} {\boldsymbol d}  {\boldsymbol \delta}\Big)  \otimes \Id_{\mathcal E^\ell}  \\
&&+2((n-k)  \alpha_{k,s}  +k \beta_{k,s} )  \Id_{\mathcal E^k}  \otimes  \Big(  \alpha_{\ell,t} {\boldsymbol \delta}   {\boldsymbol d}  + \beta_{\ell,t} {\boldsymbol d}  {\boldsymbol \delta}\Big)  .
\end{eqnarray*}
The second term of the operator $ E_{s,t}^{k,\ell}$ in \autoref{estkl}  can be rewritten as:
\begin{eqnarray*}
&&  - {2 }  \sum_{j=1}^n (x_j -y_j)  
  \Big(s  \alpha_{k,s}  {\partial \over  {\partial {x_j }}}   + \gamma_{k,s}   {\boldsymbol \varepsilon}_{j} {\boldsymbol \delta}   + \delta_{k,s} {\boldsymbol d}  {\boldsymbol \iota}_{j}\Big) \otimes 
  \Big(  \alpha_{\ell,t}  {\boldsymbol \delta}   {\boldsymbol d}  + \beta_{\ell,t} {\boldsymbol d}  {\boldsymbol \delta} \Big)      \\
&&-4\sum_{j=1}^n  
\Big(s\alpha_{k,s}   {\partial \over  {\partial {x_j }}}   + \gamma_{k,s}   {\boldsymbol \varepsilon}_{j} {\boldsymbol \delta}   + \delta_{k,s} 
{\boldsymbol d}  {\boldsymbol \iota}_{j}\Big) \otimes \Big(-2(n-2\ell)   ( {\boldsymbol \varepsilon}_{j} {\boldsymbol \delta}  
 - {\boldsymbol d}  {\boldsymbol \iota}_{j})+  \alpha_{\ell,t}  {\partial \over  {\partial {y_j }}}  \Big) \\
 &&-   {2 } \big(ns \kappa_{k,s} + (n-k)\gamma_{k,s}+ k\delta_{k,s}\big)\Id_{\mathcal E} \otimes 
 \Big(  \alpha_{\ell,t }{\boldsymbol \delta}   {\boldsymbol d}  + \beta_{\ell,t}  {\boldsymbol d}  {\boldsymbol \delta} \Big).
\end{eqnarray*}
Finally, the third   term of the operator $ E_{s,t}^{k,\ell}$ in \autoref{estkl}  can be rewritten as:
\begin{eqnarray*} 
&&2 \sum_{j=1}^n  (x_j - y_j )    \Big(  \alpha_{k,s}  {\boldsymbol \delta}  {\boldsymbol d}  +  \beta_{k,s}  {\boldsymbol d}  {\boldsymbol \delta} \Big) 
  \otimes 
     \Big(t \alpha_{\ell,t}   {\partial \over  {\partial {y_j }}}  + \gamma_{\ell,t}  {\boldsymbol \varepsilon}_{j} {\boldsymbol \delta}    + \delta_{\ell,t} {\boldsymbol d}  {\boldsymbol \iota}_{j}\Big)   \\
     &&-4\sum_{j=1}^n  
      \Big( -2(n-2k)    ( {\boldsymbol \varepsilon}_{j} {\boldsymbol \delta}  
 - {\boldsymbol d}  {\boldsymbol \iota}_{j})+  \alpha_{k,s}   {\partial \over  {\partial {x_j }}}  \Big) 
\otimes  \Big(t \alpha_{\ell,t}  {\partial \over  {\partial {y_j }}}   + \gamma_{\ell,t}   {\boldsymbol \varepsilon}_{j} {\boldsymbol \delta}   + \delta_{\ell,t} 
{\boldsymbol d}  {\boldsymbol \iota}_{j}\Big)  \\
&&- 2\big( nt \kappa_{\ell,t}+  (n-\ell) \gamma_{\ell,t}+ \ell \delta_{\ell,t} \big) 
 \Big(  \alpha_{k,s}  {\boldsymbol \delta}   {\boldsymbol d}  + \beta_{k,s}  {\boldsymbol d}  {\boldsymbol \delta} \Big)  \otimes \Id_{\mathcal E^\ell}.  
 \end{eqnarray*} 

It is worthwhile noting   that by the   anti-commutator laws  in  \autoref{tikk1}  we may rewrite the fourth term  of the operator $ E_{s,t}^{k,\ell}$  in \autoref{estkl}  as: 
$$ {{2 } }   \sum_{j=1}^n \Big( s  \alpha_{k,s} {\partial \over  {\partial {x_j }}}   + \gamma_{k,s}    {\boldsymbol \varepsilon}_{j }  {\boldsymbol \delta}+ \delta_{k,s}  {\boldsymbol d}  {\boldsymbol \iota}_{j } \Big) \otimes \Big( t\alpha_{\ell, t}  {\partial \over  {\partial {y_j }}}   + \gamma_{\ell,t}   {\boldsymbol \varepsilon}_{j } {\boldsymbol \delta}  + \delta_{\ell,t} {\boldsymbol d}  {\boldsymbol \iota}_{j }\Big) .$$

 It remains to sum up all terms to finish the poof of Theorem \ref{normal}

We close this section by writing the operator $E_{s,t}^{k,\ell}$ in  the particular case   $k=\ell=0.$ Here the operator $E_{s,t}^{0,0}$ will act on the space $\mathcal S\mathcal E^{0,0}(\mathbb R^n\times \mathbb R^n)= \mathcal S (\mathbb R^n\times \mathbb R^n).$  Since $ {\boldsymbol \delta}  =0$   and ${\boldsymbol \iota}_{j }=0$   on scalar functions, the operators $\boxvoid_{0,s} $ and $\nabla_{0,s,j}$ reduce to 
   $$\boxvoid_{0,s}  =   \alpha_{0,s}  {\boldsymbol \delta}   {\boldsymbol d}  =-\alpha_{0,s} 
Q\big({\partial \over {\partial   x}}\big)  ,$$ (see \autoref{hodge}), and $$\nabla_{0,s,j}  =(2-s) \alpha_{0,s} {{\partial  }\over {\partial x_j}}.$$  Similar identities hold with respect to $y$.  Hence
\begin{eqnarray*}
  E_{s,t}^{0,0}&=& 
  \alpha_{0,s} \alpha_{0,t}  \Big\{-\Vert x-y\Vert^2 Q\big({\partial \over{\partial x}}\big)\otimes Q\big({\partial \over{\partial y}}\big) - 2  \sum_{j=1}^n (t-2) (x_j-y_j) Q\big({\partial \over{\partial x}}\big)\otimes {\partial \over{\partial y_j}}\\
&& + 2  \sum_{j=1}^n ( s-2) (x_j-y_j)  {\partial \over{\partial x_j}} \otimes Q\big({\partial \over{\partial y}}\big)-  (t-2)(t-n)Q\big({\partial \over{\partial x}}\big) \otimes  \Id_{\mathcal E} \\
&& + 2   \sum_{j=1}^n ( s-2)( t-2) {\partial \over{\partial x_j}}\otimes {\partial \over{\partial y_j}} -   (s-2)(s-n) \Id_{\mathcal E} \otimes Q\big({\partial \over{\partial y}}\big)    \Big\}\\
&=& 
  (s+n )(s-n-2) (t+n )(t-n-2) \times\\
 &&\Big\{-\Vert x-y\Vert^2 Q\big({\partial \over{\partial x}}\big)\otimes Q\big({\partial \over{\partial y}}\big) 
 - 2  \sum_{j=1}^n (t-2) (x_j-y_j) Q\big({\partial \over{\partial x}}\big)\otimes {\partial \over{\partial y_j}}\\
&&  
+ 2  \sum_{j=1}^n ( s-2) (x_j-y_j) {\partial \over{\partial x_j}} \otimes Q \big({\partial \over{\partial y}}\big)
-  (t-2)(t-n) Q\big({\partial \over{\partial x}}\big) \otimes  \Id_{\mathcal E} \\
&&+ 2   \sum_{j=1}^n ( s-2)( t-2)  {\partial \over{\partial x_j}}\otimes {\partial \over{\partial y_j}}-   (s-2)(s-n) \Id_{\mathcal E} \otimes Q\big({\partial \over{\partial y}}\big)   
 \Big\}.
\end{eqnarray*}
Up to the normalization constant    $ (s+n )(s-n-2) (t+n )(t-n-2)$ and the change of variables $s$ by $2s$ and $t$ by $2t,$ the operator $E_{s,t}^{0,0}$ coincides  with the differential operator obtained in \cite[Proposition 10.3]{bck} to build   covariant bi-differential operators under the (diagonal) action of  the Lie group ${\rm O}(n+1,1).$


\section{Background on the   conformal group of $\mathbb{R}^n$}
Let  $ \mathbb R^{1,n+1}$  be the $n+2$-dimensional real vector space equipped with the Lorentzian quadratic form $$[{\bf x},{\bf x}] = x_0^2-(x_1^2+\dots+x_{n+1}^2),\qquad {\bf x}=(x_0,x_1,\ldots, x_{n+1}).$$
Let $\Xi$ be the isotropic cone defined by 
$$\Xi = \{ {\bf x}\in \mathbb R^{1,n+1}\setminus \{0\}\;:\;  [{\bf x},{\bf x}] = 0\}.$$
For ${\bf x}\in \mathbb R^{1,n+1}\setminus\{0\}$, denote by $[{\bf x}] = \mathbb R^* {\bf x}$ the ray through ${\bf x}$ and consider  the space of isotropic rays, i.e. the quotient space $\Xi/\mathbb R^*$. 

The subspace $\{ {\bf x}\in \mathbb R^{1,n+1}\;:\; x_0=0\}$ will be identified with $\mathbb R^{n+1}$ under the isomorphism
$$\mathbb R^{n+1} \ni x' \longmapsto (0, x')\in \mathbb R^{1,n+1} .$$
Denote by  $S^n$   the unit sphere of $\mathbb R^{n+1}.$ The map
$$S^n\ni x'\longmapsto \mathbb{R}^*(1,x')$$
yields an isomorphism of $S^n$ with $\Xi/\mathbb R^*$; the inverse isomorphism being described by
$$\Xi/\mathbb R^*\ni\mathbb R^*{\bf x} \longmapsto \mathbb R^*{\bf x} \cap \{ x_0 = 1\}.$$

Let $G={\rm SO}_0(1,n+1)$ be the connected component of the identity in the group of isometries for the Lorentzian form on $\mathbb R^{1,n+1}$. Then $G$ acts on $\Xi$ and commutes with  the action of $\mathbb R^*$ on $\mathbb R^{1,n+1}$, so that $G$ acts on $\Xi/\mathbb R^*$    and yielding an action of $G$ on $S^n.$ 

Let us give more details on the action of $G $ on the unit sphere $S^n.$ For $x'=(x'_1, \ldots, x'_{n+1})\in S^n$ and $g\in G,$ observe that $(g(1,x'))_0>0$ and define $g(x')\in S^n$ by 
$$(1, g(x'))= (g(1,x'))_0^{-1}\, g(1,x').$$ 
 For $g\in G$ and $x'\in S^n,$ set 
\begin{equation}\label{ccc} 
c(g,x') =(g(1,x'))_0^{-1}.  
\end{equation}
 Clearly $ c(g,x')$ is a smooth and strictly positive function on $G\times S^n.$ Moreover,  the function $c$ satisfies the cocycle property $$ c(g_1g_2,x') =c\big(g_1,g_2(x')\big) c(g_2,x'),\qquad 
 g_1,g_2\in G, \; x'\in S^n.$$ 
This action turns out to be \emph{conformal} on $S^n$, i.e. for any $g\in G$,   $x'\in S^n$ and arbitrary $\xi\in T_xS^n$, the differential $Dg(x')$ satisfies 
$$\Vert Dg(x') \xi \Vert = c(g,x') \Vert \xi\Vert,$$
and the term  $c(g,x')$ is called the {\em conformal factor} of $g$ at $x'.$

Let $e_{n+1} = (0,0,\dots, 0, 1)$, and let $$\kappa : \mathbb R^n\longrightarrow S^n \smallsetminus \{-e_{n+1}\}$$ 
defined by $$ (x_1,\dots, x_n) \longmapsto \left(\frac{2x_1}{1+\vert x\vert^2}\,,\cdots, \frac{2x_n}{1+\vert x\vert^2}, \frac{1-\vert x\vert^2}{1+\vert x\vert^2} \right)
$$
be the inverse map of the stereographic projection. The action of $G$ on $S^n$ can be transferred to a  rational   action (not everywhere defined) on $\mathbb R^n$, for which we still use the notation $G\times \mathbb R^n\ni(g,x) \longmapsto g(x)\in \mathbb R^n$.

The map $\kappa$ is   conformal  and hence, the rational action of $G$ on $\mathbb R^n$ transferred from its action on $S^n$ is conformal. For $g\in G$ defined at $x\in \mathbb R^n$, denote by $  \Omega(g,x)$ the corresponding conformal factor.  Below,  among  other things, we will find an expression for $  \Omega(g,x).$

Choose $e_{n+1} = (0,0,\dots, 0, 1)$ as origin on the sphere $S^n$, and let $[(1,e_{n+1})]$ be the corresponding isotropic ray. The stabilizer of $[(1,e_{n+1})]$ is a parabolic subgroup $P$ of $G$, which has the Langlands decomposition $P=MAN$ with 
$$\begin{array}{ll}
 M &=\left  \{\begin{pmatrix} 1&&\\&m&\\&&1\end{pmatrix} \;:\; m\in SO(n)\right\},\\
 &\\
A &= \left\{ a_t := \begin{pmatrix}\cosh t &&\sinh t\\ \\&\id_n&\\ \\\sinh t& &\cosh t\end{pmatrix}\;:\;  t\in \mathbb R\right\},\\
&\\
N &= \left\{n_x := \begin{pmatrix} 1+\frac{1}{2} \Vert x\Vert^2&x^t&-\frac{1}{2}\Vert x\Vert^2\\ \\x&\id_n&-x\\ \\\frac{1}{2}\Vert x\Vert^2&x^t&1-\frac{1}{2} \Vert x\Vert^2\end{pmatrix}\;:\;x\in \mathbb R^n\right\}.
\end{array}
$$
 Denote by $\overline N$  the opposite nilpotent subgroup,
$$\overline N = \left\{\overline n_x := \begin{pmatrix} 1+\frac{1}{2} \Vert x\Vert^2&x^t&\frac{1}{2}\Vert x\Vert^2\\ \\x&\id_n&x\\ \\-\frac{1}{2}\Vert x\Vert^2&-x^t&1-\frac{1}{2} \Vert x\Vert^2\end{pmatrix}\;:\; x\in \mathbb R^n\right\}.$$

The origin $e_{n+1}$ on the sphere $S^n$ corresponds to the point ${\bf 0}=(0,0,\dots, 0)$ in $\mathbb R^n$, and hence the parabolic subgroup $P$ is the stabilizer of $\bf 0$. The group $M\simeq {\rm SO}(n)$ acts on $\mathbb R^n$ by its natural action   and $A$ acts on $\mathbb R^n$ by  $$a_t(x) = e^{-t} x,\qquad a_t\in A.$$ 
The group $\overline N$ acts on $\mathbb R^n$ by translations, 
$$\overline n_y(§x) = x+y,\qquad \overline n_y\in \overline N.$$
The explicit action of $N$ on $\mathbb R^n$ (which is rational) will not be needed, but it is easily verified that
$$Dn_y({\bf 0}) = \Id_n,\qquad n_y\in N.$$

Up to a closed subset of null Haar measure, the group $G$ is equal to $\overline N\times M\times A\times N$. The corresponding decomposition of $g\in G$ is   $$g = \overline n(g) m(g) a_{t(g)} n(g).$$
An elementary computation gives 
\begin{equation}\label{ili} Dg({\bf 0}) = \Ad m(g)({\bf 0})_{\vert \overline{\mathfrak n}}\circ \Ad a_t(g)_{\vert \overline{\mathfrak n}}= e^{-t(g)}\, m(g),\end{equation}
where $\overline{\mathfrak{n}}=\mathrm{Lie}(\overline{N})$.  

Now let $x\in \mathbb R^n$ and let $g\in G$ be defined at $x$. Since $g(x)  = g\overline n_x({\bf 0})$, it follows from \autoref{ili} that 
\begin{equation}\label{Dg}
Dg(x) = D(g\overline n_x)({\bf 0}) = e^{-t(g\overline n_x)}\, m(g\overline n_x).
\end{equation}
As a consequence, we obtain 
\begin{equation}\label{Omega}
\Omega(g,x) = e^{-t(g\overline n_x)},\qquad g\in G,\, x\in \mathbb R^n.
\end{equation}

We close this paragraph by the following standard result.
\begin{lemma}  
 Let $x,y\in \mathbb R^n$ and let $g\in G$ be defined at $x$ and $y$. Then
\begin{equation}\label{cov1} 
\Vert g(x) -g(y)\Vert^2 = \Omega(g,x) \, \Vert x-y\Vert^2\,  \Omega(g,y).
\end{equation}
\end{lemma}

\section{The principal series representations of ${\rm SO}_0(1,n+1)$ on the space of differential forms}

Let $\mathcal E^k(S^n)$ be the space of differential $k$-forms on the unit sphere $S^n$ ($0\leq k\leq n$). For $\lambda\in \mathbb C,$   let $\rho^{k}_{\lambda}$ be the representation of $G={\rm SO}_0(1,n+1)$ on $\mathcal E^k(S^n)$
given by
\begin{equation*}
\rho_\lambda^k(g) \omega\,(x') = c(g^{-1},x')^\lambda \left(L_{g^{-1}}^*\omega\right) (x'),\qquad  g\in G,\; \omega\in \mathcal E^k(S^n),
\end{equation*}
where  $L_g$ is the diffeomorphism $x'\longmapsto g(x')$ on $S^n$ and $L_g^*$ is the induced action on differential forms. Here $c(g^{-1},x')$ is conformal factor given by \autoref{ccc}.

Below we will describe   the \emph{noncompact model}  for this series of representations, obtained from the present model through the stereographic projection. 

  Denote by  $\mathcal D\mathcal E^{k}(\mathbb R^n)$   the space of $k$-forms  represented as in \autoref{repkfo}    with the complex-valued coefficients $\omega_{i_1,\ldots, i_k}$ in $\mathcal D(\mathbb R^n).$    

Now, for $g\in G$ and $\omega \in \mathcal E^k(\mathbb R^n)$ let 
\begin{equation}\label{piNC}
{  \pi}^k_\lambda(g) \omega(x) =   \Omega (g^{-1}, x)^\lambda L^*_{g^{-1}} \omega (x),
\end{equation}
where $\Omega (g^{-1}, x)$ is given by \autoref{Omega}. 
This formula defines formally  a representation of $G$. As it stands, the representation is not globally defined. In what follows, it will be enough to observe that for  a relatively compact open subset $U$ of $\mathbb R^n$, there exists a small neighborhood $ V$ of  the neutral element in $G$ (depending on $U$) such that for any $g\in   V$, $g^{-1}$ is defined on $U$. Hence, for any smooth differential $k$-form $\omega$ with $Supp(\omega)\subset U$, the object   $ \pi^{ k}_\lambda(g)\, \omega$ is well defined, it belongs to $\mathcal E^k(\mathbb R^n)$ and has a compact support. This allows us to define the corresponding infinitesimal representation ${\rm d}\pi_{\lambda}^k$ of the Lie algebra $\mathfrak g = \Lie(G)$ by
$$ 
{\rm d}\pi_{\lambda}^k(X) \omega :=  {{\rm d}\over {{\rm d}t}}\pi_{\lambda }^k(\exp(tX))\omega_{\big | t=0},\qquad X\in \mathfrak g.
 $$
The expression is well defined when $\omega\in \mathcal D\mathcal E^k(\mathbb R^n).$  The operator ${\rm d}\pi_{\lambda}^k(X)$ is a first order differential operator with polynomial coefficients and hence can be extended to $\mathcal S\mathcal E^k(\mathbb R^n)$. 
 
The  representations $\pi_{\lambda}^k$   can also be viewed as principal series representations. Indeed, rewrite \autoref{piNC} as
$$ \pi_\lambda^k(g) \omega (x) = \Omega(g^{-1},x)^\lambda \omega(g^{-1}(x))\circ Dg^{-1}(x),$$
where    $\omega(g^{-1}(x))\circ Dg^{-1}(x)$ is  the $k$-form given by 
$$\omega(g^{-1}(x))\circ Dg^{-1}(x)(v_1,   \ldots, v_k)\linebreak =\omega(g^{-1}(x))(Dg^{-1}(x)v_1, \ldots, Dg^{-1}(x)v_k).$$ Now using  \autoref{Dg} and \autoref{Omega} we  get
\begin{equation}\label{spin}  \pi_\lambda^k(g) \omega (x) = e^{-(\lambda+k)t(g^{-1}\overline n_x)}\sigma_k \big(m(g^{-1}\overline n_x)\big)^{-1} \omega(g^{-1}(x) ),\end{equation}
where $\sigma_k$ is the   representation of $M={\rm SO}(n)$ on  $\Lambda^k=\Lambda^k(\mathbb R^n)\otimes\mathbb{C}$. The  presentation \autoref{spin}  is just the noncompact realization of a principal series representation 
(cf \cite{kn}). This yields   the identification
\begin{equation}\label{piPS} 
  \pi_\lambda^k\simeq \Ind_P^G (\sigma_k\otimes \chi_{\lambda+k}\otimes 1),
\end{equation}
 where, for $\lambda\in\mathbb{C}$, we denote by $\chi_\lambda$ the character of $A$ giver by $\chi_\lambda(a_t)=e^{\lambda t}$.

We pin down  that the representation $\sigma_k$ is an irreducible representation of ${\rm SO}(n)$, except for the case where $n$ is even and $k=\frac{n}{2}$   (see \cite{fjs, Ikeda-Taniguchi}),  but for our purpose, this makes no difference.

The {intertwining Knapp-Stein operators} play a crucial role   in semi-simple harmonic analysis. In the present situation they are given as follows (see \cite{fo}),
$$ I^k_\lambda \omega( x) = \int_{\mathbb R^n} \mathcal R_{-2n+2\lambda}^k(x-y)\, \omega(y) dy,\qquad \omega\in \mathcal S\mathcal E^k(\mathbb R^n)$$ 
where $\mathcal R_{-2n+2\lambda}^k$ is the tempered distribution defined by \autoref{riezf2}.  In the notations of the previous section, $ I^k_\lambda$ is nothing but the convolution operator $J_s^k$, defined in \autoref{conv-op}, with $s= {-2n+2\lambda}.$
The operators $I^k_\lambda$, 	defined   first for ${n\over 2}<\Rel \lambda<n$   so that the integral  converges for $\omega \in \mathcal S\mathcal E^k(\mathbb R^n)$, can be analytically continued to the   complex $\lambda$-plane as   a meromorphic family of convolution operators by tempered distributions, thus mapping $\mathcal S\mathcal E^k(\mathbb R^n)$ into  $\mathcal S'\mathcal E^k(\mathbb R^n)$. The following (a priori formal) relation holds for any $g\in G$ : 
\begin{equation}\label{covI}
  I^k_\lambda \circ   \pi^k_\lambda(g) =  \pi_{n-\lambda}^k(g)\circ   I^k_\lambda\ .
\end{equation}
The relation is first proved when  ${n\over 2}<\Rel \lambda<n$,  and shown  (using the covariance property \autoref{cov1} of $\Vert x-y\Vert^2$)  to be valid for forms in $\mathcal D\mathcal E^k(\mathbb R^n)$ and $g$ in an appropriate small neighborhood of the neutral element of $G$.  The   corresponding infinitesimal form of the intertwining relation \autoref{covI} is 
\begin{equation}\label{galta}
  I_\lambda^k \circ {\rm d}\pi_{\lambda}^k(X)  = {\rm d}\pi_{n-\lambda}^k(X)\circ   I_\lambda^k,
\end{equation}
for $X\in \mathfrak g$ and valid on $\mathcal S\mathcal E^k(\mathbb R^n)$. By analytic continuation, it is then extended meromorphically in $\lambda$.

The following property  will be required  later on.
\begin{proposition}\label{injectif} For generic   $\lambda$,  the operator $  I_\lambda^k$ is injective on the space $\mathcal S\mathcal E^k(\mathbb R^n)$.
\end{proposition}

\begin{proof} Since $I_\lambda^k$ is the convolution product with the   tempered distribution $\mathcal R_{-2n+2\lambda}^k,$ then saying that 
$I_\lambda^k$ is injective is equivalent to prove that (for generic  $\lambda$) the multiplication operator by the Fourier transform $\mathcal F\big(\mathcal R^k_{-2n+2\lambda} \big)$ is injective on  $\mathcal S\mathcal E^k(\mathbb R^n)$. Recall from \autoref{fo2} that, generically in $\lambda$, we have
$$\mathcal F\big(\mathcal R^k_{-2n+2\lambda} \big) (x)= 2  \Vert x\Vert^{n-2\lambda-2} \Big( ( n- k- \lambda) \boldsymbol \iota_x \boldsymbol \varepsilon_x +( \lambda- k)\boldsymbol \varepsilon_x \boldsymbol \iota_x\Big).$$ Recall also from Section 2 that 
\begin{equation}\label{rcal} 
\begin{aligned}
&(\boldsymbol \iota_x \boldsymbol \varepsilon_x)^2   =  \Vert x\Vert^2 \boldsymbol \iota_x \boldsymbol \varepsilon_x,&&&&&&
 ( \boldsymbol \varepsilon_x\boldsymbol \iota_x )^2   =  \Vert x\Vert^2  \boldsymbol \varepsilon_x \boldsymbol \iota_x, \\
 &(\boldsymbol \iota_x \boldsymbol \varepsilon_x)(\boldsymbol \varepsilon_x \boldsymbol \iota_x)  =  0 , &&&&&& (\boldsymbol \iota_x \boldsymbol \varepsilon_x)^2 + (\boldsymbol \varepsilon_x \boldsymbol \iota_x)^2  = \Vert x\Vert^4\Id_\Lambda.
\end{aligned}
\end{equation}
If we assume, in addition, that    $n-k-\lambda\neq 0$ and $\lambda-k\neq 0$, then the identities  \autoref{rcal} imply that 
 for $x\neq {\bf 0}$ the operator
$$(n-k-\lambda) \boldsymbol \iota_x \boldsymbol \varepsilon_x + (\lambda-k) \boldsymbol \varepsilon_x \boldsymbol \iota_x$$ is invertible. Let $\omega\in \mathcal S\mathcal E^k(\mathbb R^n)$ and $\lambda$ as above so that  
$ \mathcal F(\mathcal R^k_{-2n+2\lambda})(x) \omega (x)= 0$. As $ \mathcal F(\mathcal R^k_{-2n+2\lambda})(x)$ is invertible for $x\neq  {\bf 0}$, it follows that $\omega(x) =0$ for $x\neq  {\bf 0}$, and therefore $\omega\equiv 0 $ on $\mathbb R^n.$
\end{proof}

\section{The covariance property of the source operator}

We can now start to give the conformal interpretation of Theorem \ref{main2}. On one hand, we saw in the previous section that the convolution operators $J_{s}^k$ are related to the Knapp-Stein intertwining operators. So it remains to understand the conformal property of the multiplication by $\Vert x-y\Vert^2$.


The group $G={\rm SO}_0(1,n+1)$ acts rationally on the space $\mathbb R^n\times \mathbb R^n$ by the diagonal extension of its action on $\mathbb R^n$, and hence on $\mathcal E^{k,\ell}(\mathbb R^n\times \mathbb R^n)$, giving a realization  of the tensor product representation $  \pi_\lambda^{k} \otimes   \pi_\mu^{\ell}$. More explicitly
$$\pi_\lambda^{k} \otimes   \pi_\mu^{\ell}(g)\, \omega(x,y) = \Omega(g^{-1},x)^\lambda \,\Omega(g^{-1},y)^\mu\, L^*_{g^{-1}}\omega\, (x,y).$$
 
Define the multiplication operator $M: \mathcal E^{k,\ell}(\mathbb R^n\times \mathbb R^n)\longrightarrow \mathcal E^{k,\ell}(\mathbb R^n\times \mathbb R^n)$ by  
$$M\omega(x,y) = \Vert x-y\Vert^2 \,\omega(x,y).$$
The covariance property  \autoref{cov1} of $\Vert x-y\Vert^2$  immediately implies the following result.

\begin{proposition}\label{covM}  
The operator $M$ satisfies 
\[M\circ \big(  \pi_\lambda^{k} \otimes   \pi_\mu^{\ell}\big)(g) = \big(  \pi_{\lambda-1}^{k} \otimes   \pi_{\mu-1}^{\ell}\big)(g)\circ M\ .
\]
\end{proposition}
Here again the relation is valid when applied to differential  forms in $\mathcal D\mathcal E^{k,\ell}(\mathbb R^n\times \mathbb R^n)$ and $g$ in a small enough neighborhood of the neutral element of $G$. The rigorous infinitesimal version reads as follows : For every $X\in \mathfrak g$, we have
\begin{equation*}
M\circ {\rm d}\!\left(\pi_\lambda^{k} \otimes   \pi_\mu^{\ell}\right)(X)={\rm d}\!\left(\pi_{\lambda-1}^{k} \otimes   \pi_{\mu-1}^{\ell}\right) (X )\circ  M,
\end{equation*}
where, by definition, 
$${\rm d}\!\left(\pi_\lambda^{k} \otimes   \pi_\mu^{\ell}\right)(X ):= {\rm d}\pi_\lambda^{k} (X)\otimes \id+\id\otimes {\rm d}\pi_\lambda^{\ell}(X) .$$

Recall from above that the Knapp-Stein intertwining operator $I_\lambda^k$ is nothing but the convolution operator $J_s^k$ with $s=-2n+2\lambda.$ For convenience let  
$$F_{\lambda, \mu}^{k,\ell}:= -E_{n-2\lambda, n-2\mu}^{k,\ell},$$
where $E_{s,t}^{k,\ell}$ is the differential operator \autoref{estkl2}.  Explicitly 
\begin{eqnarray*}\label{F-lambda-mu-kl2}
  F_{\lambda,\mu}^{k,\ell}&=&    16\Vert x-y\Vert^2        \widetilde \boxvoid_{k,\lambda}\otimes  \widetilde \boxvoid_{\ell,\mu} \nonumber\\ 
  &&- 32 \sum_{j=1}^n (x_j-y_j)   \left\{ (2\lambda-n+2)  \widetilde \nabla_{k,\lambda,j}   \otimes  \widetilde \boxvoid_{\ell,\mu}
  -    (2\mu-n+2)  \widetilde \boxvoid_{k,\lambda} \otimes     \widetilde\nabla_{\ell,\mu,j}  \right\} \nonumber\\ 
  &&- 32 (2\lambda-n+2)(2\mu-n+2)  \sum_{j=1}^n     \widetilde \nabla_{k,\lambda,j}   \otimes      \widetilde \nabla_{\ell,\mu,j}\nonumber \\
  &&-32   (2\mu-n+2)(\mu+1)(\mu-\ell)(\mu-n+\ell)     \widetilde \boxvoid_{k,\lambda} \otimes  \Id_{{\mathcal E}^\ell}\nonumber\\
&&  -32    (2\lambda-n+2)(\lambda+1)(\lambda-k)(\lambda-n+k)      \Id_{{\mathcal E}^k} \otimes  \widetilde \boxvoid_{\ell,\mu},\nonumber
 \end{eqnarray*}
 where
 \begin{eqnarray*}
   \widetilde \boxvoid_{k,\lambda}&=&(\lambda-n+k)(\lambda-k+1) \cdd\dd+(\lambda-n+k+1) (\lambda-k) \dd\cdd \\
  \widetilde  \nabla_{k,\lambda,j} &=&(\lambda-n+k)(\lambda-k+1)\partial_{x_j}-(\lambda-k)\boldsymbol \varepsilon_j\cdd-(\lambda-n+k)\cdd \boldsymbol \iota_j ,
 \end{eqnarray*}
and similarly for $   \widetilde \boxvoid_{\ell,\mu}$ and $  \widetilde \nabla_{\ell,\mu,j}$.
 We will call $F_{\lambda, \mu}^{k,\ell}$ the \emph{source operator}. \\

Theorem \ref{main2} can now be reformulated as follows.
\begin{theorem}\label{41} The differential operator $F_{\lambda, \mu}^{k,\ell}$   acts on $\mathcal E^{k,\ell}(\mathbb R^n\times \mathbb R^n)$  and satisfies
\begin{equation*}
 \kappa_{\lambda,\mu}^{k,\ell}\;M\circ \big(  I^{k}_\lambda \otimes   I^{\ell}_\mu\big) = \big(   I^{k}_{\lambda+1} \otimes   I^{\ell}_{\mu+1}\big)\circ F_{\lambda, \mu}^{k,\ell},
\end{equation*}
where  $\kappa_{\lambda,\mu}^{k,\ell}=16(\lambda-k+1)(\lambda-n+k+1)(\mu-\ell+1)(\mu-n+\ell+1)$.
\end{theorem}

Moreover, we have the following covariance property of the source  operator $F_{\lambda, \mu}^{k,\ell}.$
\begin{theorem} For all $\lambda,\mu\in \mathbb C$ and for any $X\in \mathfrak g,$ we have 
\begin{equation*}
F^{k,\ell}_{\lambda, \mu} \circ {\rm d}\!\left(\pi_\lambda^{k} \otimes   \pi_\mu^{\ell}\right)(X ) =  {\rm d}\!\left(\pi_{\lambda+1}^{k} \otimes   \pi_{\mu+1}^{\ell}\right)(X )\circ F^{k,\ell}_{\lambda, \mu}.
\end{equation*}
\end{theorem}
 \begin{proof} In the light of Theorem \ref{41}, Proposition \ref{covM} and the identity \autoref{galta}, we have 
\begin{align*}
& \big(   I^{k}_{\lambda+1} \otimes   I^{\ell}_{\mu+1}\big)\circ F_{\lambda, \mu}^{k,\ell} \circ {\rm d}\!\left(\pi_\lambda^{k} \otimes   \pi_\mu^{\ell}\right)(X )\\
&\qquad\qquad\qquad  =  \kappa_{\lambda,\mu}^{k,\ell}  M\circ \big(   I^{k}_{\lambda } \otimes   I^{\ell}_{\mu }\big)  \circ {\rm d}\!\left(\pi_\lambda^{k} \otimes   \pi_\mu^{\ell}\right)(X )\\
&\qquad\qquad\qquad=  \kappa_{\lambda,\mu}^{k,\ell}  M\circ  {\rm d}\!\left(\pi_{n-\lambda}^{k} \otimes   \pi_{n-\mu}^{\ell}\right)(X ) \circ \big(  I^{k}_\lambda \otimes   I^{\ell}_\mu\big)\\
&\qquad\qquad\qquad = \kappa_{\lambda,\mu}^{k,\ell}  {\rm d}\!\left(\pi_{n-\lambda-1}^{k} \otimes   \pi_{n-\mu-1}^{\ell}\right)(X )\circ M\circ \big(   I^{k}_\lambda \otimes   I^{\ell}_\mu\big)\\
&\qquad\qquad\qquad  ={\rm d}\!\left(\pi_{n-\lambda-1}^{k} \otimes   \pi_{n-\mu-1}^{\ell}\right)(X ) \circ \big(   I^{k}_{\lambda+1} \otimes   I^{\ell}_{\mu+1}\big)\circ F_{\lambda, \mu}^{k,\ell}\\
&\qquad\qquad\qquad  =\big(   I^{k}_{\lambda+1} \otimes  I^{\ell}_{\mu+1}\big)\circ {\rm d}\!\left(\pi_{ \lambda+1}^{k} \otimes   \pi_{ \mu+1}^{\ell}\right)(X )\circ F_{\lambda, \mu}^{k,\ell}.
\end{align*}
Now, use  Proposition \ref{injectif} to finish the proof.
\end{proof}

As $G$ is connected, the infinitesimal covariance property of the operator $F_{\lambda, \mu}^{k,\ell}$ implies first its covariance under the group $G$, that is
\begin{equation*}
F^{k,\ell}_{\lambda, \mu} \circ \big(   \pi^{k}_\lambda \otimes   \pi^{\ell}_\mu\big)(g)  \omega = \big(  \pi_{\lambda+1}^{k} \otimes   \pi_{\mu+1}^{\ell}\big)(g)\circ F_{\lambda, \mu}^{k,\ell}\omega,
\end{equation*}
for $\omega\in \mathcal D \mathcal E^{k,\ell}(\mathbb R^n\times \mathbb R^n)$ and $g\in G$ is such that $g^{-1}$ is defined on a neighborhood of the support of $\omega$. 

\begin{remark} Let us mention that improving on our results, it is possible to construct  a differential operator ${\widetilde F}_{\lambda, \mu}^{k,\ell}$  on $S^n\times S^n$, which admits $F^{k,\ell}_{\lambda, \mu}$ as its  local expression on 
$S^n\smallsetminus{\{(0,0,\dots,0,-1)\}}\times S^n\smallsetminus{\{(0,0,\dots,0,-1)\}}\simeq \mathbb R^n \times \mathbb R^n$
and which is covariant for $G$ with respect to $(\rho_\lambda^k\otimes \rho_\mu^\ell, \rho_{\lambda+1}^k\otimes \rho_{\mu+1}^\ell)$. We skip the proof as  this corresponds to general standard results. See for instance  Section 8.2 in \cite{bck} or Fact 3.3 in \cite{kkp}.

\end{remark}
It is possible to compose the  source operators,  to yield more covariant differential operators. Indeed,  for arbitrary integer $m\geq 1$, we set   
$$ F^{k,\ell}_{\lambda, \mu;m} := F^{k,\ell}_{\lambda+m-1, \mu+m-1}\circ \dots \circ F^{k,\ell}_{\lambda+1, \mu+1}\circ F^{k,\ell}_{\lambda, \mu}\,.$$
Then $F^{k,\ell}_{\lambda, \mu;m}$ intertwines the representations $  \pi^{k}_\lambda\otimes   \pi^{\ell}_\mu$ and $  \pi^{k}_{\lambda+m}\otimes   \pi^{\ell}_{\mu+m}$.\\

 The following statement gives another approach to these operators.
\begin{proposition} For any integer $m\geq 1$, we have
\begin{equation*}
 \kappa_{\lambda,\mu; m}^{k,\ell} M^{m}\circ \big(   I^{k}_\lambda \otimes   I^{\ell}_\mu\big) = \big(   I^{k}_{\lambda+m} \otimes  I^{\ell}_{\mu+m}\big)\circ F^{k,\ell}_{\lambda, \mu;m}\,,
\end{equation*}
where $ \kappa_{\lambda,\mu; m}^{k,\ell}=\kappa_{\lambda+m-1,\mu+m-1}^{k,\ell}\cdots \kappa_{\lambda+1,\mu+1}^{k,\ell} \kappa_{\lambda,\mu}^{k,\ell}$ and $M^m=M\circ\cdots\circ M$, $m$-times.
\end{proposition}
\begin{proof} For $m=1$, this is Theorem \ref{41}. Assume $m\geq 2.$ By induction on $m$ we have 
\begin{align*}
   \kappa_{\lambda,\mu; m}^{k,\ell} M^m \circ \big(   I^{k}_\lambda \otimes   I^{\ell}_\mu\big) &=   \kappa_{\lambda,\mu; m}^{k,\ell} M\circ  M^{m-1} \circ \big(   I^{k}_\lambda \otimes   I^{\ell}_\mu\big) \\
&=  \kappa_{\lambda+m-1,\mu+m-1}^{k,\ell} M\circ   (  I^{k}_{\lambda+m-1} \otimes   I^{\ell}_{\mu+m-1})\circ  F^{k,\ell}_{\lambda,\mu;m-1} \\
&=  \big(   I^{k}_{\lambda+m} \otimes   I^{\ell}_{\mu+m}\big)\circ F^{k,\ell}_{\lambda+m-1, \mu+m-1} \circ F^{k,\ell}_{\lambda,\mu;m-1} \\
&= \big(   I^{k}_{\lambda+m} \otimes   I^{\ell}_{\mu+m}\big)\circ  F^{k,\ell}_{\lambda, \mu;m}\,.
\end{align*}
\end{proof}

\section{Conformally covariant bi-differential operators on differential forms}

From the source operators $F_{\lambda, \mu}^{k,\ell}$, one can  construct covariant bi-differential operators under the action of $G={\rm SO}_0(1,n+1).$ First, introduce the \emph{restriction map}
$$\res : \mathcal E^{k,\ell} ( \mathbb R^n\times \mathbb R^n)\longrightarrow  C^\infty ( \mathbb R^n, \Lambda^k \otimes \Lambda^\ell ) $$
defined by
$$(\res \omega) (x) =  \omega(x,x),$$
where $C^\infty ( \mathbb R^n, \Lambda^k \otimes \Lambda^\ell)$ denotes the space of complex-valued smooth functions on $\mathbb R^n$ with values in $\Lambda^k \otimes \Lambda^\ell.$
Let $G$ acts on  $\mathcal E^{k,\ell} ( \mathbb R^n\times \mathbb R^n)$ by $ \pi_\lambda^{k}\otimes  \pi_\mu^{\ell}$. 
Using the realization of $ \pi_\lambda^{k}$ and $ \pi_\mu^{\ell}$ as principal series  representations (see \autoref{piPS}), the following result is immediate.
\begin{proposition} For any $(\lambda, \mu)$ the map $\res$ intertwines the representations $\pi_\lambda^{k}\otimes  \pi_\mu^{\ell}$ and $\Ind_P^G \big((\sigma_k\otimes \sigma_\ell)\otimes \chi_{\lambda+\mu+k+\ell}\otimes 1\big)$.
\end{proposition}

As a representation of $M ={\rm SO}(n)$, the representation $\sigma_k\otimes \sigma_\ell$ is in general not irreducible. Let $\Gamma $ be a minimal invariant subspace of $\Lambda^k \otimes \Lambda^\ell $ under the action of ${\rm SO}(n)$. Let $\sigma_\Gamma$ be the corresponding irreducible representation of ${\rm SO}(n)$ on $\Gamma $ and let
$p_\Gamma$ be the orthogonal projection on $\Gamma $. Define the map $\res_\Gamma$ by
$$\res_\Gamma = p_\Gamma \circ\res.$$
We can refine the previous proposition as follows.
\begin{proposition}
For any $(\lambda, \mu)$ the map $\res_\Gamma$ intertwines the representations $\pi_\lambda^{k}\otimes  \pi_\mu^{\ell}$ and $\Ind_P^G \big(\sigma_\Gamma\otimes \chi_{\lambda+\mu+k+\ell}\otimes 1\big)$.
\end{proposition} 
Now define the  bi-differential operators $$ B_{\lambda, \mu;m}^{k,\ell;\Gamma}: \mathcal E^{k,\ell} ( \mathbb R^n\times \mathbb R^n)\longrightarrow   
C^\infty(\mathbb R^n, \Gamma )$$  by
$$
B_{\lambda, \mu;m}^{k,\ell;\Gamma} := \res_\Gamma \circ\, F_{\lambda, \mu;m}^{k,\ell}\,,
$$
where $C^\infty(\mathbb R^n, \Gamma )$ denotes the space of   smooth functions on $\mathbb R^n$ with values in $\Gamma \subset \Lambda^k \otimes \Lambda^\ell.$

\begin{theorem} The operator $B_{\lambda, \mu;m}^{k,\ell;\Gamma}$ is a   bi-differential operator covariant with respect to  $\pi_\lambda^{k}\otimes  \pi_\mu^{\ell}$ and $\Ind_P^G (\sigma_\Gamma\otimes \chi_{\lambda+\mu+k+\ell+2m}\otimes 1)$.
\end{theorem}

In some cases, it is possible to give an explicit  expression for these covariant bi-differential operators. For instance,  assume that $0\leq k+\ell\leq n$, then the representation $\Lambda^{k+\ell}$ appears in the decomposition of the tensor product $\Lambda^k \otimes \Lambda^\ell $ with multiplicity one and the projection (up to a normalization factor) is given by
$$p_{\Lambda^{k+\ell}}(\omega\otimes \eta) = \omega\wedge \eta.$$
For $m=1$, the bi-differential operator $B_{\lambda, \mu;1}^{k,\ell;\Gamma}$ is given by
 \begin{eqnarray}\label{dernier}
  B_{\lambda, \mu;1}^{k,\ell;\Gamma}(\omega\otimes \eta) (x) &=&     
  - 32 \Bigl\{ (2\mu-n+2)(\mu+1)(\mu-\ell)(\mu-n+\ell)     \widetilde \boxvoid_{k,\lambda}\omega(x)\wedge\eta(x)    \nonumber\\
  &&+(2\lambda-n+2)(2\mu-n+2)  \sum_{j=1}^n     \widetilde \nabla_{k,\lambda,j}\omega(x)\wedge       \widetilde \nabla_{\ell,\mu,j}\eta(x)  \nonumber\\ 
&&  +  (2\lambda-n+2)(\lambda+1)(\lambda-k)(\lambda-n+k)     \omega(x) \wedge \widetilde \boxvoid_{\ell,\mu}\eta(x) \Big\}
 \end{eqnarray}
 If in addition $k=\ell=0,$ i.e. $\omega,\eta \in C^\infty(\mathbb R^n),$ then 
\begin{multline*}
\shoveright{B_{\lambda, \mu;1}^{0,0;\mathbb C} (\omega\otimes \eta) (x) = 
-64(\lambda+1)(\lambda-n)(\mu+1)(\mu-n)  \Big\{\mu (\mu-\frac{n}{2}+1)     \Big(Q\big({\partial \over {\partial x}}\big) \omega\Big)(x) \eta(x)}\\
\shoveright{+2(\lambda-\frac{n}{2}+1) (\mu-\frac{n}{2}+1) \sum_{j=1}^n \Big({{\partial \omega}\over {\partial x_j}}\Big)(x) \Big({{\partial \eta}\over {\partial x_j}}\Big)(x) }\\
+  \lambda (\lambda-\frac{n}{2}+1)    \omega(x)  \Big(Q\big({\partial \over {\partial x}}\big) \eta\Big)(x)    \Big\},
\end{multline*}
   where $Q$ is the quadratic form \autoref{qua}. Hence we recover the  multidimensional Rankin-Cohen operators   in \cite[Section 10]{bck}.

\end{document}